\documentclass[reqno,12pt,letterpaper]{amsart}
\usepackage{amsmath,amssymb,amsthm,graphicx,mathrsfs,url}
\usepackage[usenames,dvipsnames]{color}
\usepackage[colorlinks=true,linkcolor=Red,citecolor=Green]{hyperref}
\usepackage{amsxtra}

\def\?[#1]{\textbf{[#1]}\marginpar{\Large{\textbf{??}}}}

\def\smallsection#1{\smallskip\noindent\textbf{#1}.}
\let\epsilon=\varepsilon 

\newcommand{\RR}{{\mathbb R}}
\newcommand{\TT}{{\mathbb T}} 

\newcommand{\SP }{{\mathbb S}}
\newcommand{\CC}{{\mathbb C}}
\newcommand{\ZZ}{{\mathbb Z}}

\newtheorem{theo}{Theorem}
\newtheorem{prop}{Proposition}[section]

\newtheorem{lemm}[prop]{Lemma}

\numberwithin{equation}{section}

\DeclareMathOperator{\Div}{div}

\DeclareMathOperator{\Spec}{Spec}

\DeclareMathOperator{\Ell}{ell}
\DeclareMathOperator{\Hol}{Hol}

\let\Im=\Imag

\let\Re=\Real

\DeclareMathOperator{\supp}{supp}
\DeclareMathOperator{\vol}{vol}
\DeclareMathOperator{\WF}{WF}
\DeclareMathOperator{\tr}{tr}

\def\WFh{\WF_h}

\title{Stochastic stability of Pollicott--Ruelle resonances}
\author{Semyon Dyatlov}
\email{dyatlov@math.mit.edu}
\address{Department of Mathematics, Massachusetts Institute of Technology,
Cambridge, MA 02139, USA}
\author{Maciej Zworski}
\email{zworski@math.berkeley.edu}
\address{Department of Mathematics, University of California,
Berkeley, CA 94720, USA}

\begin{document}

\maketitle

\begin{abstract}
Pollicott--Ruelle resonances for chaotic flows are the characteristic frequencies of correlations.
They are typically defined as eigenvalues of the generator of the flow acting on specially
designed functional spaces. We show that these resonances can be computed as viscosity limits
of eigenvalues of second order elliptic operators. These eigenvalues are the characteristic frequencies
of correlations for a stochastically perturbed flow. 
\end{abstract}

\section{Introduction and statement of results}

We consider an Anosov flow $\varphi_t=e^{tV}$ on a 
compact manifold $ X $. For the Laplacian $\Delta_g \leq 0$ with respect to some
metric on $X$, we define 
\begin{equation}
\label{eq:Peps}
P_\varepsilon={1\over i}V+i\varepsilon\Delta_g,
\end{equation}
For $ \epsilon \neq 0 $ this operator is elliptic and hence has a discrete 
$L^2 ( X ) $-spectrum $ \{ \lambda_j ( \epsilon ) \}_{j=0}^\infty $.
However, for $\epsilon=0$ most of the $L^2$ spectrum is not discrete.

Following the seminal work of Ruelle \cite{Rue}
and Pollicott \cite{Pol}, many authors investigated the discrete
spectrum of $ P_0 $ acting on specially designed \emph{anisotropic
Sobolev spaces} and the role of that spectrum in the expansion of
correlations~-- see Blank--Keller--Liverani~\cite{bkl}, Baladi--Tsujii~\cite{b-t},
Faure--Sj\"ostrand~\cite{fa-sj}, Faure--Tsujii \cite{fa-ts,fa-ts2}, 
Gou\"ezel--Liverani~\cite{g-l}, Liverani~\cite{liv2}, Tsujii~\cite{Tsu10,Ts}
and references given there. 
We review a yet another approach based on~\cite{dz} in~\S\ref{s:defPR}.
These complex eigenvalues of $P_0$, $ \{ \lambda_j \}_{j=0}^\infty $, are called 
{\em Pollicott--Ruelle resonances}. 
For perspectives on physical manifestations of these resonances see
for instance Gaspard--Ramirez~\cite{gara} or Chekroun et al \cite{cheka}.

The purpose of this note is to show that Pollicott--Ruelle resonances
can be defined as limits of $ \lambda_j ( \epsilon )$ as $ \epsilon \to 0+$.
This can be considered a {\em stochastic stability} of resonances:

\begin{theo}
\label{th:1}
Let $ P_\epsilon $ be given by \eqref{eq:Peps} and let 
$\{\lambda_j ( \epsilon ) \}_{j=0}^\infty $ be the set of its $ L^2$-eigenvalues.
If $ \{ \lambda_j \}_{j=0}^\infty $
is the set of the Pollicott--Ruelle resonances of the flow $ \varphi_t $, then
\[  \lambda_j ( \epsilon ) \longrightarrow \lambda_j , \ \ \epsilon \to 0+, \]
with convergence uniform for $ \lambda_j $ in a compact set.
\end{theo}
%
\begin{figure}
\includegraphics[scale=0.9]{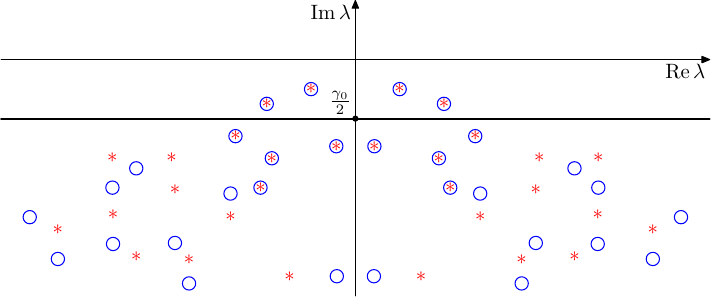}
\caption{A schematic presentation of the results in Theorems \ref{th:1}
and \ref{th:2}. Pollicott--Ruelle resonances of the generator of
the flow $ V $ (denoted by red asterisks) are approximated by 
the eigenvalues of $ V/i + i \epsilon \Delta_g $ (denoted by blue circles)
uniformly on compact
sets. The asymptotic resonance free strip is uniform with respect to 
$ \epsilon $.}
\label{f:1}
\end{figure}
The nature of convergence is much more precise 
-- see \S \ref{s:stPR}.
In particular the spectral projections depend smoothly on $ \epsilon
\in [0, \epsilon_0] $ where $ \epsilon_0 $ depends on the compact set.
Also, when 
$ \lambda_j $ is a simple resonance then for $ \epsilon $ sufficiently 
small the map $ \epsilon \mapsto \lambda_j (\epsilon )  $ is smooth all the way up to $\varepsilon=0$.
As explained in the next paragraph
 $ \lambda_j ( \epsilon ) \to \bar \lambda_j $ when $ \epsilon \to 
0 - $. We also note the symmetry of $ \lambda_j ( \epsilon )$'s with 
respect to the imaginary axis (see Fig.~\ref{f:1}). That follows from 
the fact 
that $\overline{P_\epsilon u}=-P_\epsilon\bar u$ and thus
\begin{equation}
\label{eq:reality}
\overline{(P_\epsilon-\lambda)^{-1}u}=-(P_\epsilon+\bar\lambda)^{-1}\bar u.
\end{equation}

The proof of Theorem~\ref{th:1} relies on the fact that $P_\epsilon-\lambda$ is a Fredholm operator
on the same anisotropic Sobolev spaces as $P_0-\lambda$, in a way which is controlled uniformly
as $\epsilon\to 0+$. This Fredholm property is established by the same methods
as those used in~\cite{dz} for the case of $\epsilon=0$. The key feature of the damping
term $i\varepsilon\Delta_g$ is that its imaginary part is nonpositive and thus
the propagation of singularities theorem of Duistermaat--H\"ormander
(see~\eqref{e:hyperbolic-est})
still applies in the forward time direction. For $\varepsilon<0$, the damping term
is nonnegative and propagation of singularities applies in the negative time direction,
which means that we have to consider the dual anisotropic Sobolev spaces
$H_{-sG(h)}$ and the spectrum of $P_0$ on these spaces is given by $\{\bar\lambda_j\}$.

We remark that all the results of this paper are valid for 
the operators acting on sections of vector bundles arising in dynamical 
systems~-- see \cite{dz}. We consider the scalar case to make the 
notation, which is all that is affected, simpler.

Previously, stability of Pollicott--Ruelle resonances has been 
established for Anosov maps, $ f : \TT^d  \to \TT^d $, \cite{bkl},\cite{liv2},
 following a very general argument of Keller--Liverani \cite{kl}. In that case the Koopman 
operator $ f^* :  C^\infty ( \TT^d ) \to C^\infty ( \TT^d ) $ is replaced
by a ``noisy propagator" $ G_\epsilon \circ f^* $, where $ G_\epsilon u 
= g_\epsilon * u $, $ g_\epsilon \to \delta_0 $, $ \epsilon \to 0 $. 
For general Anosov maps on compact manifolds a semiclassical proof 
was given by Faure--Roy--Sj\"ostrand \cite[Theorem 5]{frs}.
Further refinements concerning dependence on $ \epsilon $ can be
found in \cite[\S 8]{g-l} and interesting applications were 
obtained by  Gou\"ezel--Liverani \cite{g-l2} and 
Fannjiang--Nonnenmacher--Wo{\l}owski~\cite{fnw}. For a physics 
perspective on this see for instance Blum--Agam~\cite{bl} and
Venegeroles~\cite{Ve}.

For flows, Butterley--Liverani \cite{buli},\cite{buli2} showed 
that if a vector field depends smoothly on a parameter, then the spectrum of the transfer operator associated to the weight corresponding to the SRB measure
is smooth in that parameter.
Constantin--Kiselev--Ryzhik--Zlato\v s~\cite{GKRZ}
established that solutions to the heat equation with a large transport term
equidistribute after arbitrarily small times if and only if
the flow corresponding to the transport term is mixing; this can be viewed
as an analogue of our work for the $L^2$ spectrum on the real line instead of resonances.

A dynamical interpretation of $ \lambda_j (\epsilon )$'s can be
formulated as follows. In terms of the operator $ P_0 $ the flow,
$ x ( t ) := \varphi_{-t} ( x ( 0 ) ) $, is given by
\[  e^{ - i t P_0 } f ( x) = f (x (  t)  ) , \ \ 
\dot x ( t ) = - V_{ x( t ) }, \quad x(0)=x.  \]
For $ \epsilon > 0 $ the evolution equation is replaced by the
Langevin equation:
\[  e^{ - i t P_\epsilon  } f ( x ) = \mathbb E \left[ f (x (  t)  ) \right] , \ \ 
\dot x ( t ) = - V_{ x( t ) } + \sqrt{ 2 \epsilon } \dot B( t )  , \quad x(0)=x,\]
where $ B ( t ) $ is the Brownian motion corresponding to the metric $ g$ on $X$
(presented here in an informal way; see \cite{brown}). This explains why considering $ P_\epsilon $ corresponds
to a stochastic perturbation of the deterministic flow. See also Kifer \cite{ki}
for other perspective on random perturbations of dynamical systems. 

We also remark that a result similar to Theorem \ref{th:1} is valid for 
{\em scattering resonances}: for $ V \in L^\infty_{\rm{c}} ( \RR^n ;
\CC ) $ (and in greater generality) they appear as limits of eigenvalues
of $ - \Delta + V( x ) - i \epsilon |x|^2 $ when $ \epsilon \to 0 + $, see
\cite{resvis}. The proof is based on the method of complex scaling 
and is technically very different than the one presented here. The result
however is exactly analogous with spacial infinity, $ |x| \to \infty $,
replacing the momentum infinity, $ |\xi| \to \infty $.

\begin{center}
\begin{figure}
\includegraphics[width=3in]{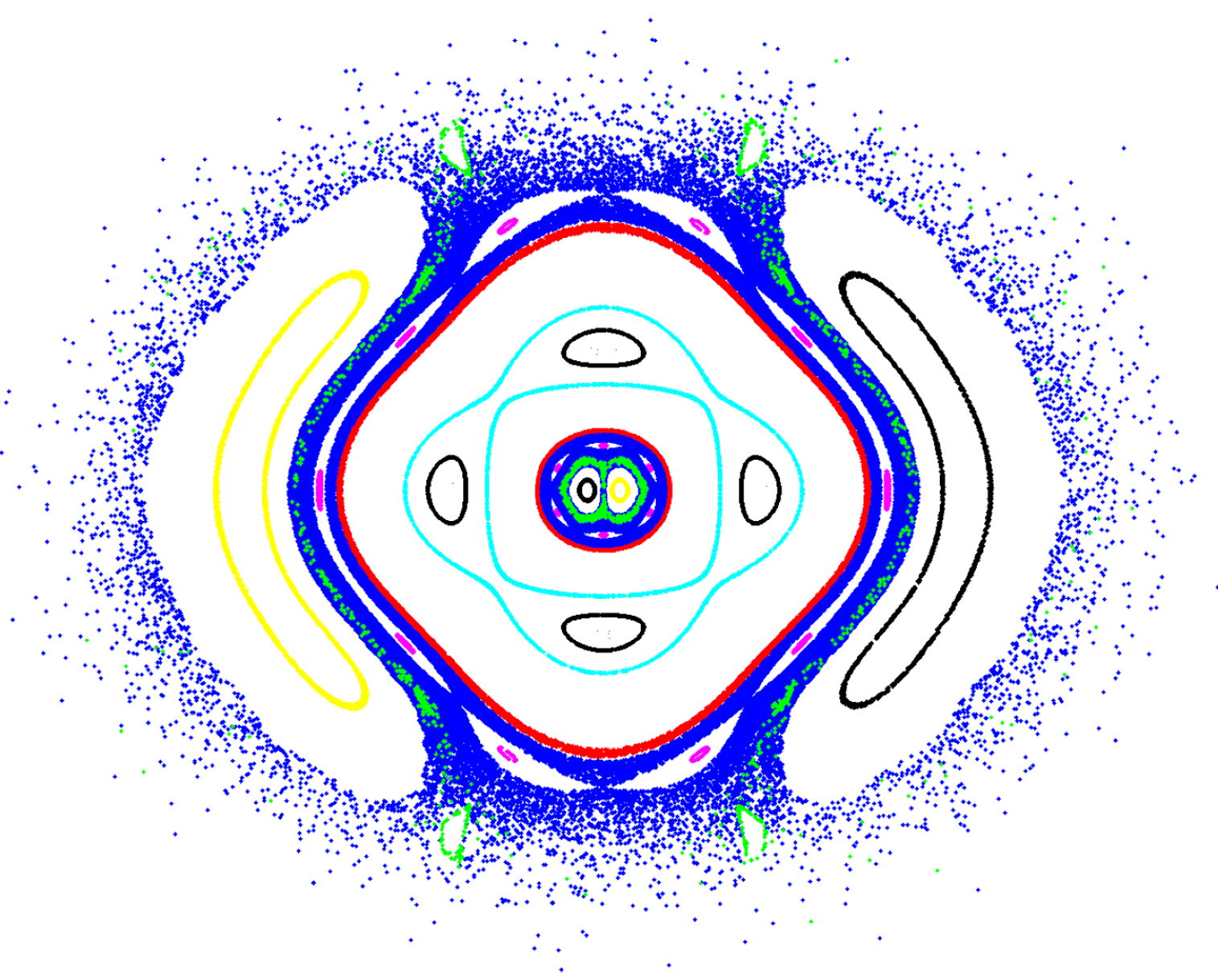} \hspace{0.2in} 
\includegraphics[width=2.5in]{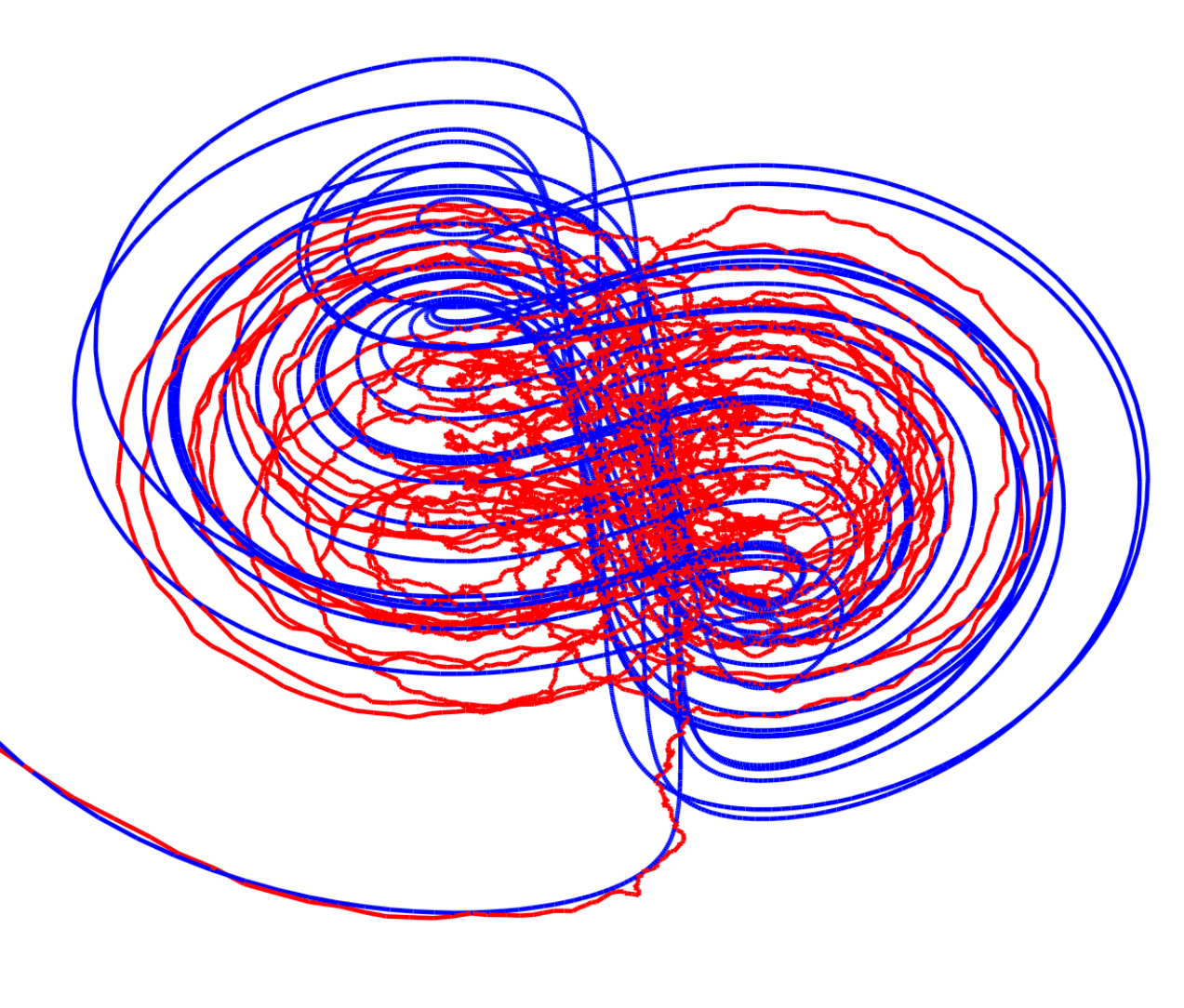} 
\caption{An illustration of
chaotic and stochastic trajectories: we consider
the Nos\'e--Hoover oscillator \cite{NoHo} which is possibly the simplest chaotic 
system:
$ W = x_2 \partial_{x_1} + ( -x_1 + x_2 x_3) \partial_{x_2} + ( 1 - x_2^2) \partial_{x_3} $, $
x \in \RR^3 $. The vector field $ V = e^{|x|^2/2} W $ is the Reeb 
vector field for the contact form $ \alpha = e^{-|x|^2/2} ( x_2 d x_1 + dx_3 ) $.
On the left the Poincar\'e section $ \{x_3 = 0\} $ showing
the chaotic sea and islands of quasi-periodicity (each colour corresponds
to a {\em numerical} iteration of a single point).  On the right 
a chaotic trajectory and the stochastic trajectory, $ \epsilon = 0.01 $,
with the same
initial data. We stress that the results of our paper do not apply 
to mixed systems and this example is meant as an illustration of
chaotic and stochastic trajectories. However, as in \cite{cheka},
Pollicott--Ruelle resonances are expected to be relevant for mixed systems
as well.}
\label{f:chaos}
\end{figure}
\end{center}

Pretending that spectrum of $ P_\epsilon $ is semisimple 
(algebraic multiplicities are equal to geometric multiplicities -- see \S \ref{s:stPR} for 
the general statement), the relation to the eigenvalues $ \lambda_j ( \epsilon ) $ 
comes from considering long time behaviour:
for any $ f \in C^\infty ( X ) $, and $ t > 0 $, 
\begin{equation}
\label{eq:t4}
\begin{split}  e^{ - i t P_\epsilon } f ( x) & =  
 \sum_{ \Im \lambda_j > - A }  
e^{ - i t \lambda_{j} ( \epsilon ) } u_{j}^\epsilon ( x ) \int_X v _{j}^\epsilon ( y)
 f ( y) 
d \!\vol_g (y) + 
\mathcal O_{f} ( e^{ - t A } )_{C^\infty } ,\end{split}
\end{equation}
where $ u_{j}^\epsilon, v_j^\epsilon \in C^\infty ( X) $ are the eigenfunctions of
$ P_\epsilon $ and $ P_\epsilon^t $ corresponding to $ \lambda_j ( \epsilon) $. 
We note that there are
no convergence problems as the number of $ \lambda_j ( \epsilon ) $
with imaginary parts above $ - A $ is finite though
the number will grow with $ \epsilon $. In fact, \cite{Jin} shows
that the number of Pollicott--Ruelle resonances, $ \lambda_j $, with 
$ \Im \lambda_j > - A $ is always infinite if $ A $ is sufficiently large.

The validity of a modification of \eqref{eq:t4} for $ \epsilon = 0 $ is only known
for {\em contact Anosov flows} (see \S \ref{s:con}) and for 
$ A > - \gamma_0/2 $, where  $\gamma_0 $ is an averaged Lyapounov
exponent (see \eqref{eq:gammd}). That is due to 
Tsujii \cite{Tsu10,Ts} who followed earlier advances by 
Dolgopyat \cite{Dol} and Liverani \cite{Liv}. It is also a 
consequence of more general results obtained  in \cite{nz5}.  

The modification in \eqref{eq:t4} is needed since the corresponding
$ u_j$'s are now distributions and the expansion provides fine
aspects of the decay of correlations. Let $ d \mu ( x ) $ be
the volume form obtained from the contact form on $ X$, $ \mu ( X ) = 1 $.
 For $ f , g \in C^\infty ( X ) $
and any $ \delta > 0 $, 
\begin{equation}
\label{eq:t4c}
\begin{split} 
&   \int_X  \left[ e^{- i t P_0} f\right] (x)
g ( x )d \mu ( x ) 
 = 
 \int_X f ( x)   d\mu ( x) \, \int_X g ( x ) d\mu ( x) \\
 & \ \ \ \ \ \ \ \ \ \ \ \ 
   + \sum_{ - \frac12(\gamma_0 - \delta) 
   < \Im \lambda_j < 0 }   e^{ - i t \lambda_{j} } 
   v_{j} ( f ) u_{j} ( g ) + \mathcal O_{f,g} ( e^{ - \frac12t(        
  \gamma_0 - \delta ) } ), 
\end{split}
\end{equation}
where $ \gamma_0 $ is 
the minimal asymptotic growth rate of the unstable Jacobian, that is the largest constant
such that for each $\delta>0$
\begin{equation}
\label{eq:gammd}
|\det (d\varphi_{-t}|_{E_u(x)})|\leq C_\delta e^{-(\gamma_0-\delta)t},\ t\geq 0;\quad  
\varphi_{-t}^* = e^{ - i t P_0 } : C^\infty ( X ) \to C^\infty ( X ) ,
\end{equation}
with $ E_u ( x ) \subset T_x X $ the unstable subspace of the flow at
$ x $ -- see \S \ref{s:pr}. Now $ u_j $ and $ v_j $ are 
distributional eigenfunctions of $ P_0 $ and $ P_0^t $, 
$ \WF ( u_j ) \subset E_u^* $ and $ \WF ( \bar v_j ) \subset E_s^* $.
Here again we make the simplifying
assumption that the spectrum is semisimple; 
that is always the case for geodesic flows in constant 
negative curvature as shown by Dyatlov--Faure--Guillarmou \cite[Theorem~3]{DFG}.

Hence it is natural to ask the question if the gap $ \gamma_0/2 $ is
uniform with respect to $ \epsilon $, that is, if the expansion 
\eqref{eq:t4} with $ A > -\frac12 ( \gamma_0 - \epsilon ) $ 
uniformly approaches the expansion \eqref{eq:t4c}. That is
indeed a consequence of the next theorem:

\begin{theo}
\label{th:2}
Suppose that $ X $ is an odd dimensional compact manifold and
that $ V \in C^\infty ( X; TX ) $ generates a contact Anosov flow.
There exists a constant $ s_0 $ 
such that for any $ \delta > 0 $ there exist $ N_0,R > 0 $ such that for all
$ \epsilon > 0 $, 
\begin{equation}
\label{eq:t2}   ( P_\epsilon - \lambda)^{-1} = {\mathcal O}( \lambda^{N_0} )
 : H^{  s_0 } ( X ) \to H^{ - s_0  } ( X ) , 
\end{equation}
for $ \gamma_0 $ defined in \eqref{eq:gammd} and 
$ \lambda \in [ R , \infty ) - i [ 0 , \textstyle{\frac12} ( \gamma_0 - \delta ) ] $. 
\end{theo}
The same estimate is true for
$\lambda\in (-\infty, -R]-i[0,\textstyle{\frac12} ( \gamma_0 - \delta ) ] $ by 
recalling \eqref{eq:reality}.
Since on the compact set $ [ - R ,  R] - i [ 0 , \textstyle{\frac12} ( \gamma_0 - \delta ) ]$,
$ \lambda_j ( \epsilon ) $ 
converge uniformly to $ \lambda_j$'s, we see that for $ \epsilon $ small
enough the number of eigenvalues of $ P_\epsilon $ in that set 
is independent of $ \epsilon $. We should remark that 
the estimate \eqref{eq:t2} can be made more precise by using 
microlocally weighted spaces reviewed in \S \ref{s:mimoa} -- see \eqref{eq:precise}.

The proof of Theorem 2 combines the approach of Faure--Sj\"ostrand \cite{fa-sj}
and \cite{dz} with the work on resonance gaps for general differential 
operators \cite{nz5}. As in that paper we also use the resolvent
gluing method of Datchev--Vasy \cite{DaVa}.

For a class of maps on $ \TT^2 $ a similar result
has been obtained by Nakano--Wittsten \cite{NaWi}.

\begin{figure}
\begin{center}
\includegraphics[scale=0.5]{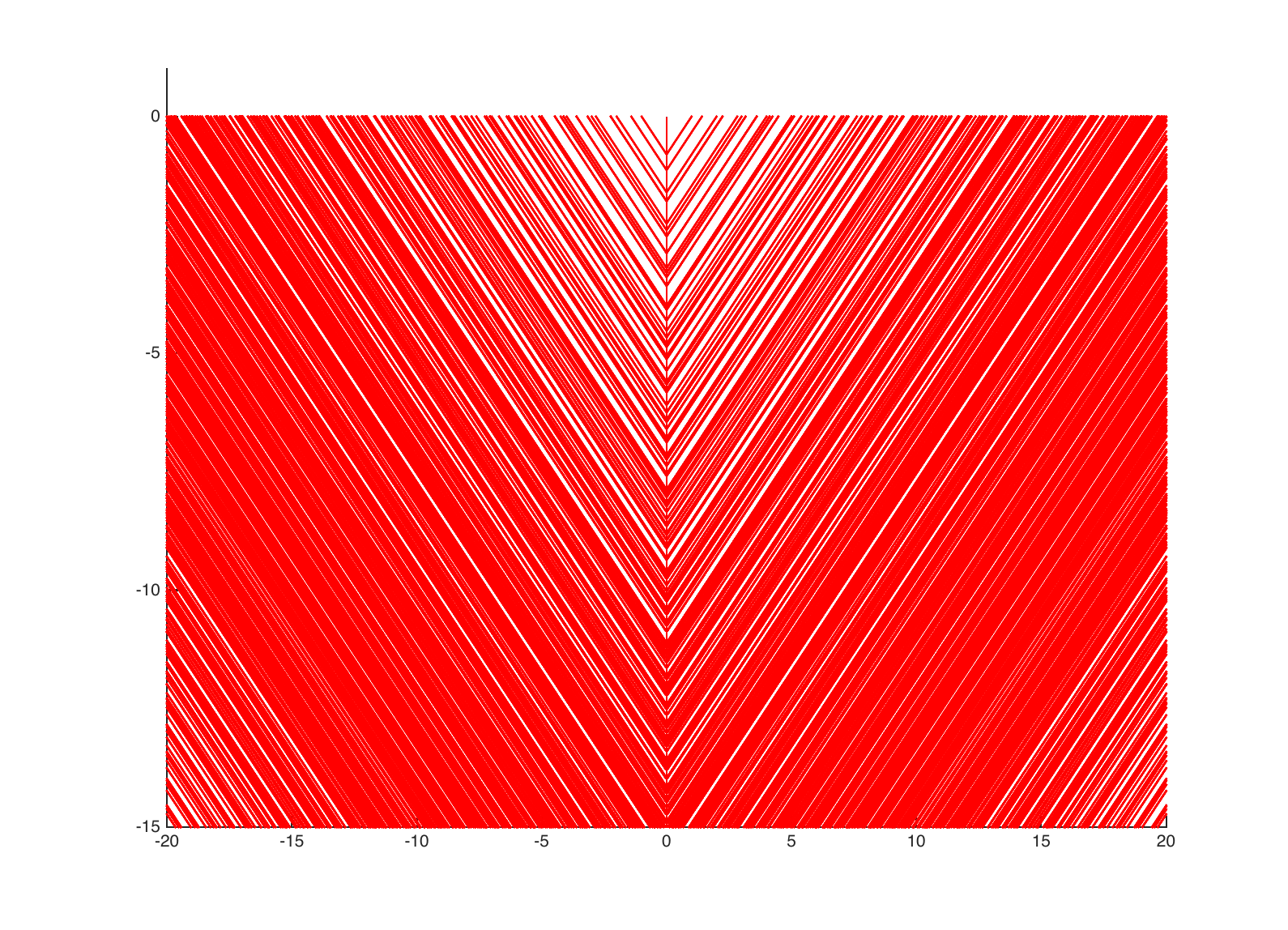}
\end{center}
\caption{The case of geodesic flow on a torus: the unit tangent bundle
is given by $ \TT^3 = \SP^1 \times \SP^1 \times \SP^1 $. If $ V $ is the
generator of the geodesic flow and $ \Delta $ is the (flat) Laplacian
on $ \TT^3 $ then
accumulation points of spectra of $ V/i + i \epsilon \Delta $ as $                          
\epsilon \to 0+ $ form a discrete set of lines.  That is dramatically
different from the Anosov case shown in Fig.~\ref{f:1}.}
\label{f:gash}
\end{figure}

\smallsection{Negative examples} It is important to point out
that the existence of a discrete limit set for the eigenvalues
of the operator $ P_\epsilon $ is very special to chaotic flows
and for mixed flows could only hold under some special domain restrictions.
The simplest ``counterexample" is given by considering $ X = \SP^1 \times
\SP^1 $ with $ V = \partial_{x_1} + \alpha \partial_{x_2} $, $ x_j \in 
\SP^1 := \RR/2 \pi \ZZ $. When $ \alpha $ is irrational then 
accumulation points of the spectrum of $ P_\epsilon $ as $ \epsilon \to 0 +$
form the lower half plane. When $ \alpha = p/q $ with $ p $ and $ q $ coprime
the limit set is equal to $ \ZZ/q - i [ 0 , \infty ) $. 

A more interesting example is given by the geodesic flow on the torus,
$ \TT^2 = \SP^1 \times \SP^1 $ with the flat metric. 
That is a contact flow on the unit cotangent bundle
$ S^* \TT^2 = \TT^2_{x_1,x_2} \times \SP^1_\theta $, generated by 
and it is generated by
\[  V = \cos \theta \partial_{x_1 } +  \sin \theta \partial_{x_2}  .\]
Defining $ P_\epsilon $ using the flat Laplacian n on $ \TT^3$, and
by expanding in Fourier modes in $x$ we see that
\[ \Spec ( P_\epsilon ) = \bigcup_{ n \in \ZZ^2 } \Spec ( P_\epsilon (n ) ) , 
\ \ P_\epsilon ( n ) := n_1 \cos \theta + n_2 \sin \theta - i \epsilon (
n_1^2 + n_2^2 + D^2_\theta ) , \]
$ D_\theta = \frac 1 i \partial_\theta $. 
We rewrite the operator $ P_\epsilon (n ) $ as follows:
\[ P_\epsilon ( n ) = - i \epsilon D_\theta^2 + |n| \cos ( \theta - \delta_n ) 
- i |n |^2 \epsilon , \ \  \delta_n = \tan^{-1} ( n_1/n_2 ) .  \]
For $ n = 0 $ the spectrum is simply $ - i \epsilon m^2 $, $ m \in \ZZ $ and
it accumulates on the negative imaginary axis. For $ n \neq 0 $ the 
asymptotic behaviour of the spectrum is determined by the asymptotic behaviour of the 
spectrum of the semiclassical operator 
\[ Q ( h ) := ( h D_\theta)^2 + i \cos \theta , \ \ h^2 := \epsilon / | n | .\]
That has been determined by Galtsev--Shafarevitch \cite{gash} who showed
that as $ h \to 0 $ the spectrum concentrates on 
on a rotated ``Y'' shape with
the vertices at $ \pm i $ and the junction at a special value $ E^* \approx 0.85 $. 

This shows that the accumulation points in the case of the generator
of the geodesic flow on the two torus regularized using the flat Laplacian
are given by 
 \[ - i [ 0 , \infty ) \cup
\bigcup_{ n \in \ZZ^2 \setminus \{ 0 , 0\} } \{ z \, : \, | \Re z | \leq n , \ \ \Im z =
- E^* |n|  + E^* | \Re z| \} , \]
and part of this set is shown in Fig.~\ref{f:gash}.

\smallsection{Acknowledgements}
We would like express thanks to Carlangelo Liverani for suggesting this
problem, to Viviane Baladi and St\'ephane Nonnenmacher for 
helpful comments on an earlier version of the paper, and 
to Michael Hitrik for informing us of the reference \cite{gash}.
We are also grateful for the support by the
Clay Research Fellowship (SD) and by
the National Science Foundation grant DMS-1201417 (MZ).

\section{Preliminaries}
\label{s:pr}

We review some definitions and basic facts mostly to fix notation and to 
provide references. The needed results from microlocal/semiclassical 
analysis are presented in detail in \cite[\S 2.3, Appendix C]{dz} 
and we will rely on the presentation given there.

\smallsection{Notation} We use the following notation: $ f =  \mathcal O_\ell ( g
 )_H $ means that
$ \|f \|_H  \leq C_\ell  g $ where the norm (or any seminorm) is in the
space $ H$, and the constant $ C_\ell  $ depends on $ \ell $. When either $ \ell $ or
$ H $ are absent then the constant is universal or the estimate is
scalar, respectively. When $ G = \mathcal O_\ell ( g )_{H_1\to H_2 } $ then
the operator $ G : H_1  \to H_2 $ has its norm bounded by $ C_\ell g $.

\subsection{Dynamical systems}
In this paper $ X $ is a compact manifold and $ \varphi_t : X \to X $ 
a $ C^\infty $ flow, $ \varphi_t = \exp t V $,  $ V \in
C^\infty ( X; T X) $.  The flow is {\em Anosov} if the tangent space 
to $ X $ has a continuous decomposition $ T_x X = E_0 ( x ) 
\oplus E_s (x) \oplus E_u ( x)$ which is invariant, 
$ d \varphi_t ( x ) E_\bullet ( x ) = E_\bullet ( \varphi_t ( x ) ) $,
$E_0(x)=\mathbb R V(x)$,
and for some $ C$ and $ \theta > 0 $ fixed 
\begin{equation}
\label{e:anosov}
\begin{split} 
&  | d \varphi_t ( x ) v |_{\varphi_t ( x )}  \leq C e^{ - \theta
  |t| } | v |_x  , \ \ v
\in E_u ( x ) , \ \ 
t < 0 , \\
& | d \varphi_t ( x ) v |_{\varphi_t ( x ) }  \leq C e^{ - \theta |t| } |  v |_x , \ \
v \in E_s ( x ) , \ \ 
t > 0 .
\end{split} 
\end{equation}
where $ | \bullet |_y $ is given by a smooth Riemannian metric on $
X $. 

Following Faure--Sj\"ostrand \cite{fa-sj} we exploit the analogy between
dynamical systems and quantum scattering, with the fiber ($\xi$) infinity 
playing the role of $x$-infinity in scattering theory. The 
pull-back map can be written analogously to the Schr\"odinger propagator
\[  \varphi_{-t} = e^{ - i t P_0 } ,  \ \ P_0 := \frac1 i V . \]
The symbol of $ P_0 $ and its Hamiltonian flow are
\[ p ( x , \xi ) = \xi( V_x ) , \ \ e^{ t H_{p} }(x,\xi) = 
 ( \varphi_t ( x ) , ( {}^T                          
d\varphi_t ( x) )^{-1} \xi )  .\]
Here $ H_p $ denotes the Hamilton vector field of $ p$: 
$ \omega ( \bullet , H_p ) = dp $, where $ \omega = d ( \xi dx ) $ 
is the symplectic form on $ T^* X $.

In the study of $ P_0 $ we need the dual decomposition of 
the cotangent space:
\begin{equation}
\label{eq:Estar} T_x^*X=E_0^*(x)\oplus E_s^*(x)\oplus E_u^*(x),
\end{equation}
where
$E_0^*(x),E_s^*(x),E_u^*(x)$ are symplectic annhilators of
$ E_s ( x ) \oplus E_u ( x ) $, $ E_0 ( x ) \oplus E_s ( x ) $,
and $ E_0 ( x ) \oplus E_u ( x )$. Hence they are dual to 
to $E_0(x),E_u(x),E_s(x)$.

A special class of Anosov flows is given by contact Anosov flows.
In that case $ X $ is a contact manifold, that is a manifold equipped
with a contact $1$-form $ \alpha $: that means that if the dimension of $X$ is 
$ 2k-1 $ then $ (d \alpha)^{\wedge k } \wedge \alpha $ is non-degenerate.
A contact flow is the flow generated by the Reeb vector field $ V$:
\begin{equation}
\label{eq:contact}      \alpha ( V ) = 1 , \ \ d \alpha ( V , \bullet ) = 0  . 
\end{equation}
For an example of a non-Anosov contact flow see Fig.\ref{f:chaos}. An 
important class of examples of Anosov contact flows is obtained from
negatively curved Riemannian manifolds $ ( M , g ) $: 
$ X = S^*M := \{ ( z, \zeta ) \in T^* M; |\zeta|_g = 1 \}$, $ \alpha = 
\zeta dz |_{S^* M } $.

\subsection{Wave front set of distributions and operators}
\label{wfs}

Semiclassical quantization on a compact manifold \cite[Appendix C]{dz},\cite[Chapter 14]{e-z} is central to our analysis.  

Let $ X $ be a compact manifold and $ h \in ( 0 , 1 ) $ a parameter
(the asymptotic parameter in the semiclassical analysis). 
A family of $h$-dependent distributions $ u\in\mathcal D'(X) $
is called  $ h$-tempered if for some $ N $, $ \| u \|_{H^{-N}} \leq 
C h^{-N} $. 
A phase space description of singularities of $ u $ 
is given by the wave front set:
\[ \WFh ( u ) \subset \overline T^* X , \]
where $\overline T^* X $ is the fiber-radially compactified cotangent 
bundle, a manifold with interior $T^*X$ and boundary, 
\begin{equation}
\label{eq:kappa}  \partial\overline T^*X=S^*X
=(T^*X\setminus 0)/\mathbb R^+, 
\ \ \kappa : T^*X \setminus 0 \longrightarrow S^* X = \partial \overline T^* X .
\end{equation}
 In addition to singularities, $\WFh$ measures oscillations on the
$h$-scale. We also refer to it as the {\em microsupport} of $ u $
or as having $ u $ {\em microlocalized} to some region containing $ \WFh ( u ) $ -- see \S C.2 for the definitions. 

For families of ($h$-tempered) operators we define the
wave front set  $\WFh'(B)$ using the Schwartz kernel of $ B $, $ K_B $:
\[  \WFh' ( B ) = \{ ( x, \xi, y , - \eta ) : ( x, y , \xi, \eta ) \in 
\WFh ( K_B ) \} . \]
This convention guarantees that $ \WFh ( I) = \Delta_{T^*X } $ is the 
diagonal, $ \{ ( x, \xi, x , \xi ) \} $, in $ T^*X \times T^*X $.

\subsection{Pseudodifferential operators}
\label{pso}

We only use the standard class of semiclassical pseudodifferential operators, 
$\Psi_h^m ( X ) $ with the symbol map $ \sigma_h $, for which 
\[ 0 \longrightarrow h \Psi^{m-1}_h ( X ) \hookrightarrow \Psi_h^m  ( X ) 
\stackrel{ \sigma_h }{\longrightarrow} S^m ( X ) / h S^{m-1} ( X ) 
\longrightarrow 0 , \]
is a short exact sequence of algebra homomorphisms and 
\[ S^m ( X ) := \{ a \in C^\infty ( T^*X ) : \partial_x^\alpha \partial_\xi^\beta a ( x, \xi ) = \mathcal O_{\alpha \beta} ( \langle \xi \rangle^{m-|\beta| } ) \} \]
(where we were informal about coordinates on $ X $).

One of our uses of the pseudodifferential calculus is that for $\chi\in C_0^\infty(\mathbb R)$, the operator
$\chi(-h^2\Delta_g)$, defined via spectral theory on $L^2$, is pseudodifferential
in the class $\Psi^{-N}_h$ for each $N$, and $\sigma_h(\chi(-h^2\Delta_g))=\chi(|\xi|_g^2)$~--
see~\cite[Theorem~14.9]{e-z}. Moreover, we implicitly use in the analysis
of the operator $\widetilde P_\epsilon(\lambda)$ in~\S\ref{s:mimoa}
that the $S^0$-seminorms of the full symbol
of $\chi(-h^2\Delta_g)$ are controlled by the $S^0(\mathbb R)$-seminorms of $\chi$.
To see that, we use the proof of~\cite[Theorem~14.9]{e-z}
to write the full symbol of $\chi(-h^2\Delta_g)$ in the form
(see~\cite[Propositions~2.2 and~2.4]{zeeman} for details)
\begin{equation}
  \label{e:symbian}
\sum_{j=0}^\infty h^j \sum_{k=0}^{2j} \chi^{(k)}(|\xi|_g^2)a_{j,k}(x,\xi),\quad
a_{j,k}\in S^{2k-j}(T^*X).
\end{equation}
If we control $\sup_{\lambda\in\mathbb R} \langle\lambda\rangle^k\chi^{(k)}(\lambda)$ for
all $k\geq 0$, then we control $\chi^{(k)}(|\xi|^2_g)$ in $S^{-2k}$ and thus we
control~\eqref{e:symbian} in $S^0 $.

The semiclassical Sobolev spaces on $ X $ are defined as
\begin{equation}
\label{eq:Sobh}
H_h^s ( X ) = ( I - h^2 \Delta_g )^{-s/2} L^2 ( X ) \subset 
\mathcal D' ( X ) ,
\end{equation}
for a choice of a Laplacian $ \Delta_g \leq 0 $ on $ X $
and with the inner product inherited from $ L^2 $.

For $ A \in \Psi^m_h ( X ) $ the elliptic set $\Ell_h(A)\subset 
\overline T^*X$ is defined
as the set of 
 $(x,\xi) \in \overline T^* X $ such $\langle \xi' \rangle^{-k}|\sigma_h(A)(x',\xi';h)| \geq c>0$
for $h$ small enough and all $(x',\xi')\in T^*X$ in a neighbourhood of $(x,\xi) $. 
We recall \cite[Proposition 2.4]{dz}:

\begin{prop}
\label{2.4}  
Suppose that $ P \in \Psi^k ( X ) $ and that $ u(h)\in \mathcal        
D'(X;\mathcal E)$ be $h$-tempered. Then 
\begin{equation}
  \label{e:elliptic-wf}
\WFh( u)\cap \Ell_h( P)\subset\WFh( Pu).
\end{equation}

If $A\in\Psi^0_h(X)$ 
and $\WFh(A) \subset \Ell_h ( P ) $, then
for each $m$,
\begin{equation}
  \label{e:elliptic-est}
\|A u\|_{H^m_h(X)}\leq C\| P 
u\|_{H^{m-k}_h(X)}+\mathcal O(h^\infty).
\end{equation}
\end{prop}

\subsection{Propagation estimates}
\label{propa}

The crucial components of the proofs of Theorems \ref{th:1} and \ref{th:2}
are propagation results presented in \cite[\S 2.3]{dz} and
proved in \cite[\S C.3]{dz}. 

\begin{figure}
\includegraphics[scale=0.9]{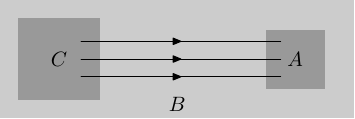}
\caption{The assumptions on the microsupports of 
the operators in the 
propagation estimate \eqref{e:hyperbolic-est} and the flow lines
of the Hamilton vector field $ H_p $. This is the simplest of the 
cases illustrated in Fig.~\ref{f} with $ A_3 $ and $ C_3 $ (in that
case $ B_3 $ is the identity).}
\label{f:pr}
\end{figure}

We start by recalling a modification of the result
of Duistermaat--H\"ormander:
\begin{prop}\label{2.5}
Assume that $ \widetilde 
P\in\Psi^1_h(X) $ and the semiclassical principal symbol, 
$
\sigma_h(\widetilde  P)\in S^1_h(X) /hS^0_h(X)
$,
has a representative $ \tilde p - i q $, where for some 
$ \delta > 0  $,
\begin{equation}
\label{eq:weirdass}
\tilde p = p + \mathcal O ( h^\delta )_{S^{1/2} ( T^*X ) }, \ \ 
p ( x, t \xi ) = t p ( x, \xi ) \in \RR  , \ \ \|\xi|_g \geq 1, \ \ t \geq 1, \ \ 
q \geq 0 .
\end{equation}
Let $e^{tH_p}$ be the Hamiltonian flow of $p$ on $\overline T^*X$
and $ u(h)\in \mathcal D'(X;\mathcal E)$ be an $h$-tempered
family of distributions.
Then (see Figure~\ref{f:pr}):

1. Assume that $A,B,C\in\Psi_h^0(X)$ and for each $(x,\xi)\in\WFh(A)$,
there exists $T\geq 0$ with $e^{-TH_p}(x,\xi)\in\Ell_h(C)$
and $e^{tH_p}(x,\xi)\in \Ell_h(B)$ for $t\in [-T,0]$.
Then for each $m$,
\begin{equation}
  \label{e:hyperbolic-est}
\|A  u\|_{H^m_h(X;\mathcal E)}\leq K\|C u\|_{H^m_h(X;\mathcal E)}
+K h^{-1}\|B   P  u\|_{H^m_h(X;\mathcal E)}+\mathcal O(h^\infty).
\end{equation}

2. If $\gamma(t)$ is a flow line of $H_p$, then for each $T>0$,
\begin{equation}
  \label{e:hyperbolic-wf}
\gamma(-T)\not\in\WFh(  u),\
\gamma([-T,0])\cap\WFh(  P  u)=\emptyset\ \Longrightarrow\
\gamma(0)\not\in\WFh(  u).
\end{equation}
\end{prop}
\begin{proof}
We explain the modifications needed in the proof of \cite[Proposition~2.5]{dz} where $ \widetilde p = p $. 
 We again construct the escape function $ f $ using 
 the homogeneous part of the symbol given by  $p$. The difference
$\widetilde p -p$ produces an additional $\mathcal O(h^\delta )_{\Psi^{2m-1/2}}$ term
in the operator $ \mathbf T_\epsilon$ of~\cite{dz}, which is uniform in the parameter $\epsilon$
of~\cite{dz} (note that in~\cite{dz} the letter $\varepsilon$
has a different meaning than in the current paper). The
$H^{m-1/2}_h$ norm should be replaced
by the $H^{m-1/4}_h$ norm on the right-hand side of~\cite[(C.12)]{dz},
which leads to the same modification on the right-hand side of~\cite[(C.5)]{dz};
the rest of the proof is carried out the same way as in~\cite{dz}.
\end{proof}

This propagation result is applied
away from the radial sinks and sources
given by $ \kappa ( E_s^* ) $ and $ \kappa ( E_u^* ) $ where $ \kappa $ is
the projection in \eqref{eq:kappa} and $ E^*_\bullet $ are from \eqref{eq:Estar}.
Near $ \kappa ( E_\bullet^* ) $ we use radial estimates obtained in the 
context of scattering theory by Melrose \cite[Propositions~9,10]{mel}
(see also Vasy~\cite[Propositions~2.3,2.4]{vasy1}). 
These less standard estimates guarantee
regularity of $  u$ near sources/sinks,
provided that $  u$ lies in a sufficiently high/low Sobolev space.

Let $p$ satisfy the assumptions in \eqref{eq:weirdass}.
Assume that $ L \subset T^* X \setminus 0 $ is a closed conic set
invariant under the flow $e^{tH_p}$. It is called 
a {\em radial source} if 
there exists an open conic neighbourhood $U$ of 
$ L$ with the following properties valid for some constant $\theta>0$:
\begin{equation}
\label{eq:defL}
\begin{split}
d\big( \kappa(e^{-tH_p}(U) ) , \kappa ( L ) \big) \to 0 
 \ &\text{ as $  t \to +\infty $;} \\
   (x,\xi)\in U   \ \Longrightarrow \ 
 |e^{-tH_p}(x,\xi)| \geq C^{-1} e^{\theta t} |\xi|,  \ &\text{ for any 
norm on the fibers.}
\end{split}
\end{equation}
 A {\em radial sink} is 
defined analogously, reversing the direction of the flow.

\begin{figure}
\includegraphics{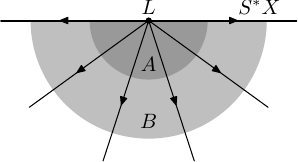}
\qquad\qquad
\includegraphics{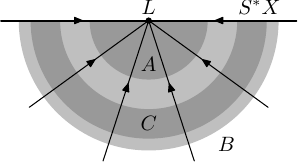}
\hbox to\hsize{\hss (a)\qquad\qquad\qquad\hss (b) \hss}
\caption{(a) The assumptions of Proposition~\ref{2.6}.
(b) The assumptions of Proposition \ref{2.7}. Here $S^*X$ is
the boundary of $\overline T^*X$ and the flow lines of $H_p$
are pictured.}
\label{f:radial}
\end{figure}

We now have a propagation estimate near radial sources.
It shows that $ P u $ controls $ u$ there for sufficiently 
regular solutions:

\begin{prop}\label{2.6}
Let $  P\in\Psi^1_h(X)$ and assume that $ \sigma_h ( P ) $
has a representative of the form $ p - i q $ 
and $ p $ and $ Q $ satisfy \eqref{eq:weirdass}.
Assume that $L\subset T^*X\setminus 0$ is a radial source 
for the flow of $ H_p $. Then there exists $m_0>0$ such that
(see Figure~\ref{f:radial}(a))

1. For each $B \in\Psi^0_h(X)$ elliptic on $\kappa(L)\subset S^*X=\partial \overline T^*X$, there exists $A\in\Psi^0_h(X)$
elliptic on $\kappa(L)$ such that if
$  u(h)\in\mathcal D'(X;\mathcal E)$ is $h$-tempered, then for each $m\geq m_0$,
\begin{equation}
  \label{e:radial1-est}
A  u\in H^{m_0}_h\ \Longrightarrow\
\|A  u\|_{H^m_h}\leq K h^{-1}\|B   P   u\|_{H^m_h}+\mathcal O(h^\infty).
\end{equation}

2. If $  u(h)\in\mathcal D'(X;\mathcal E)$ is $h$-tempered and
 $B \in\Psi^0_h(X)$ is elliptic on $\kappa(L)$, then
\begin{equation}
  \label{e:radial1-wf}
B   u\in H^{m_0}_h,\
\WFh(  P  u)\cap \kappa(L)=\emptyset\ \Longrightarrow\
\WFh(  u)\cap\kappa(L)=\emptyset.
\end{equation}
\end{prop}

The second result shows that for sufficiently low
regularity we have a propagation result at radial sinks analogous to
\eqref{e:hyperbolic-est}.

%
\begin{prop}\label{2.7}
Assume that $  P\in\Psi^1_h(X)$ is as in Proposition~\ref{2.6}
and $L\subset T^*X\setminus 0$ is a radial sink. Then there exists $m_0>0$ such that
for each $B \in\Psi^0_h(X)$ elliptic on $\kappa(L)$, there exists
$A\in\Psi^0_h(X)$ elliptic on $\kappa(L)$ and $C \in\Psi^0_h(X)$ with
$\WFh(C)\subset\Ell_h(B )\setminus \kappa(L)$, such that if $  u(h)\in\mathcal D'(X)$
is $h$-tempered, then for each $m\leq -m_0$ (see Figure~\ref{f:radial}(b))
\begin{equation}
  \label{e:radial2-est}
\|A  u\|_{H^m_h}\leq K\|C  u\|_{H^m_h}+K h^{-1}\|B   P  u\|_{H^m_h}
+\mathcal O(h^\infty).
\end{equation}
\end{prop}

The proofs of Propositions \ref{2.6} and \ref{2.7} can be found in 
\cite[\S C.3]{dz}.

\section{Definition of Pollicott--Ruelle resonances}
\label{s:defPR}

The resonances for Anosov flows are defined as spectra of the generator of
the flow acting on suitably modified spaces -- see Baladi--Tsujii \cite{b-t},  
Faure--Sj\"ostrand \cite{fa-sj}, 
Gou\"ezel--Liverani \cite{g-l}, Liverani \cite{liv2}, and references given there.

Here we follow \cite[\S 3.1--3.2]{dz} where the spaces are defined using microlocal 
weights  with simple properties:
\begin{gather}
\label{eq:weighted}
\begin{gathered}
H_{sG(h)} (X ) := \exp ( - s G ( x, h D) ) L^2 ( X ) , \ \  G \in \Psi^{0+}_h (  X) , \\
\sigma_h ( G ) = ( 1 - \psi_0 ( x, \xi) ) m_G ( x , \xi ) \log | \xi|_g , 
\end{gathered}
\end{gather}
where $ \psi_0 \in C^\infty_{\rm{c}} ( T^*X, [ 0 , 1 ] ) $ is $ 1 $ near $\{ \xi = 0 \}$, 
$ m_G ( x, \xi ) \in C^\infty ( T^* X \setminus 0 , [ -1, 1 ] ) $ is homogeneous of
degree $ 0 $ and satisfies
\begin{equation}
\label{eq:mGdef}  m_G ( x , \xi ) = \left\{ \begin{array}{ll}  \ \ 1 & \text{near $ E_s^* $}\\
-1 & \text{near $ E_u^* $ } \end{array} \right. \ \ H_p m_G ( x, \xi ) \leq 0 , \ \
( x, \xi) \in T^* X \setminus 0 . \end{equation}
The existence of such $ m_G $ is shown in \cite[Lemma C.1]{dz}. 
For convenience we choose $ |\xi|_g^2 $ to be the same metric as in the definition 
of the Laplacian $ -\Delta_g $. We can also assume that for some 
$\chi_0 \in C^\infty_{\rm{c}} ( \RR ) $, $ \chi_0 \equiv 1 $ near $ 0 $, 
\begin{equation}
\label{eq:mGch0}
G ( x, h D ) = ( 1 - \chi_0 ( -h^2 \Delta_g ) ) G ( x, h D) . 
\end{equation}
(Simply multiply $ G ( x , h D) $ by $ ( 1 - \tilde \chi_0 ( -h^2 \Delta_g ) )$
for $ \tilde \chi_0 \in C^\infty_{\rm{c}} $ such that if $ |\xi|_g \in \supp \tilde 
\chi_0 $ then $ \psi_0 ( x, \xi ) = 1 $ and then choose $ \chi_0 $ so that
$ \supp \chi_0 \subset \tilde \chi_0^{-1} ( 1 ) $.)

We note that as a set $ H_{ s G (h) } $ is independent of $ h$ and that for some  $N $ and 
$ C $, 
\begin{equation}
\label{eq:1h}   h^N \| u \|_{ H_{s G ( 1 ) }} / C \leq \| u \|_{ H_{s G ( h ) } }  \leq C 
h^{-N} \| u \|_{ H_{s G ( 1 ) } } \,. \end{equation}

We also need a version of weighted Sobolev spaces associated to $ H_{s G( h ) } $:
\begin{gather}
\label{eq:Hrs}
\begin{gathered}   H^r_{s G ( h ) } := \exp( - G_{r,s} ( x, h D ) ) L^2 ( X) , \ \ 
G_{r,s}  \in \Psi^{0+}_h ( X ) ,  \\ 
\sigma_h ( G_{r,s} ) = ( 1 - \psi_0 ( x, \xi ) ) ( s m_G ( x, \xi ) + r ) 
\log |\xi|_g . \end{gathered}
\end{gather}
We can also assume that \eqref{eq:mGch0} holds for $ G_{s,r} $ as well. 

The spaces with $ r \neq 0 $ will be used to control applications of
differential operators:
\begin{equation}
\label{eq:Hrs1}
 \Psi^m_h ( X ) \ni A : H^r_{s G ( h ) } ( X ) \longrightarrow H^{r-m}_{s G ( h ) }
 ( X ) . 
 \end{equation}
Since (see \cite[(3.9)]{dz})
\[  H_p \sigma_h ( G_{r,s} ) = s \log | \xi|_g H_p m_G + {\mathcal O}(1)_{S^0_h},
\]
we can use the estimates reviewed in 
\S \ref{propa} as in the proof of \cite[Proposition 3.4]{dz}. That
shows that for any $ r \in \RR$, 
$ \lambda \in D ( 0 , R ) $, $ s > s_0 = s_0 ( R, r ) $ and $ 0 < h < h_0 $, 
\begin{equation}
\label{eq:Qrs}
( hP_0 - i Q - h \lambda )^{-1} = \mathcal O ( 1/h ) : H^r_{ s G ( h ) } \longrightarrow 
H^r_{s G ( h ) } .
\end{equation}
Here 
$ Q $ is  a {\em complex absorbing operator}
\begin{equation}
\label{eq:defQ}
Q=\chi(-h^2\Delta_g), \ \ 
\chi\in C_0^\infty((-2,2); [0,1]), \ \ 
\chi ( t ) =1, \ t \in [-1,1] . 
\end{equation}
It is introduced to damp the trapped set which, on $ p^{-1} ( 0 )$, is
equal to the zero section. Writing
\[ P_0 - \lambda = h^{-1}  ( I + i Q ( h P_0 - i Q - 
h\lambda  )^{-1} ) (h P_0 - i Q - h \lambda ) , \]
and noting that 
\begin{equation}
\label{eq:QtoC} Q  ( h P_0 - i Q - \lambda h )^{-1} : H_{s G ( 1 ) }  \to 
C^\infty ( X ) , 
\end{equation} 
 is compact as an operator $ H_{ s G( 1) } \to H_{ s G ( 1 ) } $, 
analytic Fredholm theory (see for instance \cite[Theorem D.4]{e-z}) shows
that $ ( P_0 - \lambda)^{-1}$ is a meromorphic family:
\begin{prop}
\label{p:defPR}
For $\lambda \in D(0,R)$ and $s > s_0=s_0(R)$, 
\[  ( P_0 - \lambda)^{-1} : H_{sG (1) } \to H_{s G ( 1 ) } , \]
is a meromorphic family of operators with poles of finite rank.
These poles are independent of $ s $ and are called {\em Pollicott--Ruelle
resonances}.
\end{prop}

The mapping property \eqref{eq:QtoC} also shows that
the operator there is of trace
class. Combined with Gohberg--Sigal theory (see for instance
\cite[(C.4.6)]{res}) this gives the following
characterization of Pollicott--Ruelle resonances:

\begin{prop}
\label{p:detPR}
Let $ R > 0 $ and assume that $ s > s_0( R )$. For $ 0 < h < h_0 ( R,s )  $ define
\[  D_R ( \lambda ) := {\det}_{ H_{s G ( 1) } } ( I +  i Q ( h P_0 - i Q - \lambda h )^{-1}  ) ,\quad
\lambda\in D(0,R). \]
Then Pollicott--Ruelle resonances in $ D ( 0 , R ) $ 
are given, with multiplicities, by the zeros of $ D_R $.
\end{prop}

\section{Microlocal bounds on the modified operator}
\label{s:mimoa}

Let $ P_\epsilon $ be given by~\eqref{eq:Peps} and let $ Q $ be
the complex absorbing operator \eqref{eq:defQ}. The goal of this
section is to prove that for $ 0 < h < h_0 $ and $ 0 \leq \epsilon < h/C $
the operator $ h P_\epsilon - i Q - h \lambda $ is invertible on the same weighted
spaces on which $ h P_0 - i Q - h \lambda $ is invertible.
Note that for $\epsilon>0$, $hP_\epsilon-iQ-h\lambda$
is a Fredholm operator $H^2_{sG(h)}\to H_{sG(h)}$ of index 0 by the standard elliptic theory
applied to the conjugation of this operator
by $e^{sG(x,hD)}$ (see~\eqref{eq:conj} below and~\cite[Theorem~19.2.1]{ho3}).

We first prove an elliptic estimate, which does not involve the parameter $h$:
\begin{lemm}
\label{l:ell}
Suppose that $\chi_1 \in C_0^\infty((-2,2), [ 0 , 1 ] )$ satisfies
$\chi_1 =1$ on $[-1,1]$, 
and put $ \chi_2 ( t ) := \chi_1 ( 3t ) $. 
Then for $ \lambda \in D ( 0 , R ) $, 
\begin{equation}
  \label{e:ell}
\begin{split}
& \|(1-\chi_1 (-\varepsilon^2 \Delta_g))u\|_{H_{sG( 1) }}
+\|\varepsilon^2\Delta_g(1-\chi_1 (-\varepsilon^2 \Delta_g))u\|_{H_{sG ( 1) }}\\
& \ \ \ \leq 
C \epsilon \| ( 1 - \chi_2 ( -\varepsilon^2 \Delta_g ) ) ( P_\epsilon - \lambda ) u \|_{H_{sG( 1) }}
+\mathcal O_R (\varepsilon^\infty)\|u\|_{H_{sG ( 1) }} .\end{split}
\end{equation}
\end{lemm}
\begin{proof}
In \eqref{eq:mGch0} we can assume that $ \supp \chi_0 \subset
\chi_2 ^{-1} ( 1 ) $: changing $ \chi_0 $ corresponds to 
changing $ \psi_0 $ in the definition of $ H_{s G ( 1 ) } $
that produces an equivalent norm 
(see \cite[Theorem 8.8]{e-z}). 

The weight of the space $H_{sG(1)}$ is not smooth at the zero section when
one considers the $\varepsilon$-quantization. To counteract this problem,
we introduce a new, $ \epsilon$-dependent, norm on $ H_{s G ( 1 ) }  $ using 
a modified weight:
\begin{equation} 
\label{eq:normse}   \| u \|_{ s , \epsilon } := \| e^{ s G_\epsilon ( x, 
\epsilon D) } u \|_{L^2}, \ \ G_\epsilon ( x, \epsilon D ) := 
( 1 - \chi_0 (- \epsilon^2  \Delta_g ) ) G ( x, D ) , \end{equation}
where 
$  G_\epsilon ( x , \epsilon D) \in \log(1/\epsilon)\Psi_\epsilon^{0+} ( X ) $ and
\begin{equation}
\label{eq:sigmaeps} \sigma_\epsilon ( G_\epsilon ( x, \epsilon D ) ) := 
( 1 - \chi_0 (|\xi|_g^2 ) ) \log ( |\xi|_g / \epsilon ) 
m_G ( x , \xi ) \mod \epsilon \log(1/\epsilon) S^{-1+} ( T^* X ) .\end{equation}
(We used here the homogeneity of $ m_G $: $ m_G ( x, \xi/ \epsilon ) =
m_G ( x , \xi ) $.)
 
We now claim that for $ j = 1, 2 $, 
\begin{equation}
  \label{eq:1eps}
(e^{sG(x,D)}-e^{sG_\epsilon(x,\epsilon D)})(1-\chi_j(-\epsilon^2\Delta_g))=\mathcal O(\epsilon^\infty)_{\mathcal D'(X)\to C^\infty(X)}.
\end{equation}
This can be rewritten as the following identity for $t=s$:
$$
(e^{tG(x,D)}e^{-tG_\epsilon(x,\epsilon D)}-I)e^{sG_\epsilon(x,\epsilon D)}(1-\chi_j(-\epsilon^2\Delta_g))
=\mathcal O(\epsilon^\infty)_{\mathcal D'(X)\to C^\infty(X)}.
$$
Differentiating the left-hand side in $t$, we obtain
$$
e^{tG(x,D)}C(t,s),\quad
C(t,s)=\chi_0(-\epsilon^2\Delta_g)G(x,D)e^{(s-t)G_\epsilon(x,\epsilon D)}(1-\chi_j(-\epsilon^2\Delta_g)).
$$
We now consider $ C ( t,s ) $ as an operator
in $ \Psi^{s-t+} ( X ) $.
Since $ \supp \chi_0 \cap \supp ( 1 - \chi_j ) = \emptyset $, 
we see that the all the terms in the symbolic composition formula 
for the four factors in $ C ( t, s ) $ vanish.
The remainder (estimated, for instance, as in \cite[(9.3.7)]{e-z})
is of size $ \epsilon^N $ for any any $ N $.  Hence 
$ C ( t,s) \in \epsilon^\infty \Psi^{-\infty } ( X ) $ and consequently
\[  e^{t G ( x, D ) } C ( t,s ) \in 
\epsilon^\infty \Psi^{-\infty }( X ) , \]
uniformly for bounded $ t,s $. Integration then gives \eqref{eq:1eps}.

By~\eqref{eq:1eps}, we may replace the $H_{sG(1)}$ norms in~\eqref{e:ell} by the
$\|\bullet\|_{s,\epsilon}$ norms.
We now consider our operator in the $\varepsilon$-pseudodifferential calculus:
\[ \varepsilon ( P_\varepsilon - \lambda ) \in \Psi_\varepsilon^2 , \ \ 
p_\varepsilon( x, \xi ) := \sigma_\epsilon ( \varepsilon P_\epsilon ) = -i | \xi|^2_g + \xi ( V_x ) .\]
This operator is elliptic in the class $\Psi^2_\epsilon$ for $\xi\neq 0$. By the choice of $\chi_j$'s, we see that both
$\varepsilon(P_\epsilon-\lambda)\in\Psi^2_\epsilon$ and $(1-\chi_2(-\epsilon^2\Delta_g))\in\Psi^0_\epsilon$
are elliptic on 
$ \WF_\varepsilon ( 1 - \chi_1 ( -\epsilon^2 \Delta_g ) )$.
Hence the estimate \eqref{e:ell} holds for $ s = 0 $ -- 
see Proposition \ref{2.4} above.

To prove~\eqref{e:ell} for the
$ \|\bullet\|_{s, \epsilon  }$-norms,
we consider conjugated operators:
\[   P_{ \epsilon, s} := e^{s G_ \epsilon ( x, \epsilon D ) } 
P_\epsilon e^{ -s G_ \epsilon( x , \epsilon D )  ) } , \ \
A_{j,s } ( x, h D ) := e^{s G_ \epsilon ( x , \epsilon D )} 
( 1 - \chi_j ( -\epsilon^2 
\Delta_g ) ) e^{ -s G_ \epsilon( x, \epsilon D ) } , \]
and need to prove that
\begin{equation}
\label{eq:ell1}
\| A_{1,s} u \|_{H_\epsilon^2 } \leq C  \| A_{2,s}(\epsilon P_{\epsilon, s } - \epsilon 
\lambda ) \|_{L^2} +\mathcal O(\epsilon^\infty)\|u\|_{L^2}. 
\end{equation}
(The conjugation of $ \epsilon^2 \Delta_g $ appearing in 
\eqref{e:ell} is handled in the same way as $ \epsilon P_{\epsilon, s}$ below.)

We have, as in \cite[\S 3.3]{dz}, $\epsilon P_{\epsilon,s}\in\Psi^2_\epsilon$, $A_{j,s}\in \Psi^0_\epsilon$, and
\begin{equation}
\label{eq:conj}  \epsilon P_{\epsilon, s} = \epsilon P_\epsilon - i \epsilon s \frac{i} {\epsilon} [ G_ \epsilon  , \epsilon P_\epsilon ]
+ \mathcal O (\epsilon^2 \log(1/\epsilon))_{ \Psi_{\epsilon}^{0+} } , \end{equation}
so that
\[ \sigma_\epsilon ( \epsilon P_{\epsilon, s } ) =- i  |\xi|_g^2 + \xi ( V ) +
\epsilon sH_{|\xi|^2_g } \sigma_\epsilon ( G_\epsilon  ) + i \epsilon H_{ \xi ( V ) } \sigma_\epsilon ( G_\epsilon ) \mod  \epsilon S^1 ( T^* X ) . \]
Recalling \eqref{eq:sigmaeps} we see that 
\[  | H_{|\xi|^2_g } \sigma_\epsilon ( G_ \epsilon  ) | +  |H_{ \xi ( V ) } \sigma_\epsilon ( G_\epsilon ) | \leq 
C \log(1/\epsilon) \langle \xi\rangle^{1 +} .\] 
Hence, $ \epsilon P_{ \epsilon , s } - \epsilon \lambda $ 
is elliptic in $ \Psi^2_\epsilon $ on the 
set $ |\xi| > \delta$ for any $ \delta > 0 $. 
Composition of pseudodifferential operators in $ \Psi^*_\epsilon $ shows
that 
\[  \WF_\epsilon ( A_{1, s } ) \subset \{ |\xi | > 1 \} \subset {\rm{ell}}_\epsilon ( \epsilon P_{ \epsilon , s}-\epsilon \lambda ), \ \ 
\WF_\epsilon ( I - A_{2,s} ) \cap \WF_\epsilon ( A_{1,s} ) = 
\emptyset .\]
We can apply Proposition \ref{2.4} again to obtain 
\eqref{eq:ell1} and hence \eqref{e:ell}.
\end{proof}
 
We turn to the question of invertibility of $ h P_\epsilon - i Q - \lambda h$
and suppose that
\[  ( h P_\epsilon - i Q - \lambda h ) u = f .\]

For $ \epsilon < h/C $ we have $ ( 1 - \chi_2 ( -\epsilon^2 \Delta_g ) ) 
Q =0 $. Hence in view of \eqref{eq:1h} and \eqref{e:ell},
\begin{equation}
\label{eq:ell2.1}
\begin{split}
& \|(1-\chi_1 (-\varepsilon^2 \Delta_g))u \|_{H_{sG( h) }}
+\|\varepsilon^2\Delta_g(1-\chi_1 (-\varepsilon^2 \Delta_g))u\|_{H_{sG ( h) }}\\
& \ \ \ \leq 
C h^{-N}  \epsilon \| ( 1 - \chi_2 ( -\varepsilon^2 \Delta_g ) ) f \|_{H_{sG( h) }} 
+\mathcal O_{R}  (h^{-N} \varepsilon^\infty)\|u\|_{H_{sG ( h) }} ,\end{split}
\end{equation}
for $ \lambda \in D ( 0 ,  R ) $, $ \epsilon < h/C $, and some $ N $ depending on 
$ s $. 

Put
\begin{equation}
\label{eq:defPeh}   \widetilde P_{\epsilon} ( \lambda ) := {h\over i}V +
i\epsilon h \Delta_g \chi_1 ( -\epsilon^2 \Delta_g ) - iQ - \lambda h ,
\end{equation}
Then 
\begin{equation}
  \label{e:houston}
\widetilde P_{\epsilon  } ( \lambda ) u  
=-i\varepsilon h\Delta_g \big(1-\chi_1(-\varepsilon^2\Delta_g)\big)u+f=:F.
\end{equation}
From \eqref{eq:ell2.1} we see immediately that
\begin{equation}
\label{eq:houston}
\|F\|_{H_{sG ( h ) }}\leq C h^{-N} \|f\|_{H_{sG ( h ) }}+\mathcal O( h^{-N} \varepsilon^\infty)\|u\|_{H_{sG ( h ) }},
\end{equation}
where $ N $ depends on $ s $.

The operator $ \widetilde P_{\epsilon} ( \lambda ) $ 
on the left-hand side of~\eqref{e:houston} is an $h$-pseudodifferential
operator in $\Psi^1_h$ and 
$$
\sigma_h ( \widetilde P_{\epsilon} ( \lambda ) ) = \xi(V_x)-i|\xi|_g\chi_1\Big({\varepsilon^2\over h^2}|\xi|^2_g\Big){\varepsilon\over h}|\xi|_g-
 i\chi(|\xi|_g^2)
 \in S^1 ( T^*X ) ,
$$
uniformly in $ \epsilon \in (0,  Ch ) $, $ \lambda \in D ( 0 , R ) $. 
The domain of this operator is given by the domain of $  V $ acting on 
$ H_{s G (h)} $: 
\begin{gather*}
 D_{ s G ( h ) } = \{ u \in H_{ s G( h ) } \mid V u \in H_{s G ( h ) } 
\subset {\mathcal D' } ( X ) \}\,, \\
\| u \|_{ D_{s G( h ) } } = \| u \|_{ H_{sG ( h )}} + h \| V u \|_{ H_{sG( h )} }. 
\end{gather*}

We now verify that the main estimate of \cite[\S 3.3]{dz} is valid for the
operator $ \widetilde P_{ \epsilon  } ( \lambda ) $. The key fact is that
the operator is now of order $ 1 $ in $ \xi $ as, using Lemma \ref{l:ell}, 
we can control $ F $ by $ f $. 
\begin{lemm}
\label{l:rad}
Suppose that $ \lambda \in D ( 0 , R ) $ and that $ 0 \leq \epsilon \leq h/C_0 $. 
Then there exist $ h_0 = h_0 ( R )$, $ s_0 = s_0 ( R ) $, $C=C(R)$
(independent of $ \epsilon $) such
that for $ u \in D_{ s G ( h ) } $
 and the operator $ \widetilde P_{ \epsilon } ( \lambda )$ defined in~\eqref{eq:defPeh}
\begin{equation}
\label{eq:rad}
\|u\|_{H_{sG(h) }}\leq Ch^{-1}\| \widetilde P_{ \epsilon } ( \lambda ) u \|_{H_{sG(h)}}, 
\ \  s_0 < s , \ \ 0 < h < h_0 .
\end{equation}
\end{lemm}
\begin{proof} We refer to the proof of \cite[Proposition 3.4]{dz} for 
details and explain the differences between the operator 
$ \widetilde P_{ \epsilon}(\lambda) $ and the operator
$ \widetilde P_0(\lambda)={h\over i} V - i Q -h\lambda$
considered there.  We recall that the proof is based
on propagation results recalled in \S \ref{propa}.

First of all,
near $ \kappa ( E_s^* ) $, where $\kappa:T^*X\setminus 0\to S^*X=\partial \overline T^*X$
is the projection to fiber infinity,
we use the {\em radial source} estimate (Proposition \ref{2.6}).
The operator $ \widetilde P_{ \epsilon } ( \lambda ) $ satisfies
the assumptions of Proposition \ref{2.6} and we get for each $ N $
\begin{equation}
\label{eq:A1}  \| A_1 u \|_{ H^s_h } \leq Ch^{-1} \| B_1\widetilde P_{ \epsilon } 
( \lambda ) u \|_{ H_h^{s } } + {\mathcal O } ( h^\infty ) \| u \|_{ H_h^{-N} }
    , \ \ s > s_0 , \end{equation}
where both $A_1,B_1\in\Psi^0_h$ are microlocalized
in a small neighborhood of $\kappa(E_s^*)$ and
$ A_1$ is elliptic near $ \kappa ( E_s^* ) $~-- see
Fig.~\ref{f}.
From the properties of the weight $ G$ -- see \eqref{eq:mGdef} -- 
we see that 
\[ \| A_1 u \|_{  H_{s G ( h ) } } = \| A_1 u \|_{ H^s_h } + \mathcal O ( h^\infty )
\| u \|_{ H_h^{-N} }  , \ \ 
\| B_1 f \|_{ H_{s G ( h ) } } = \| B_1 f \|_{ H^s_h } + \mathcal O ( h^\infty )
\| u \|_{ H_h^{-N} } , \]
and hence we can replace $ H^s_h $ by $ H_{s G ( h ) } $ in \eqref{eq:A1}.

Similarly if $ A_2\in \Psi^0_h$ is microlocalized near $ \kappa ( E_u^* ) $
there exist $ B_2, C_2\in \Psi^0_h  $ microlocalized near $ \kappa ( E_u^* ) $ with 
$ \WFh ( C_2 ) \cap \kappa( E_u^* ) = \emptyset $ (see Fig.~\ref{f})
such that 
\begin{equation}
\label{eq:A2}  \| A_2 u \|_{ H^{-s}_h } \leq C \|C_2 u\|_{H^{-s}_h}+
Ch^{-1} \| B_2 
\widetilde  P_{ \epsilon  } 
( \lambda ) u \|_{ H_h^{-s } } + {\mathcal O } ( h^\infty ) \| u \|_{ H_h^{-N} }
    , \ \ s > s_0 . \end{equation}
This follows from Proposition \ref{2.7}.
Recalling \eqref{eq:mGdef} again we see that
\[ \| A_2 u \|_{  H_{s G ( h ) } } = \| A_2 u \|_{ H^{-s}_h } + \mathcal O ( h^\infty )
\| u \|_{ H_h^{-N} }  , \ \ 
\| B_2 f \|_{ H_{s G ( h ) } } = \| B_2 u \|_{ H^{-s}_h } + \mathcal O ( h^\infty )
\| u \|_{ H_h^{-N} } , \]
and similarly for $C_2$, so that again the estimate \eqref{eq:A2} is valid with $ H_h^{-s} $ replaced
by $ H_{sG (h ) } $. 

\begin{figure}
\includegraphics[scale=0.9]{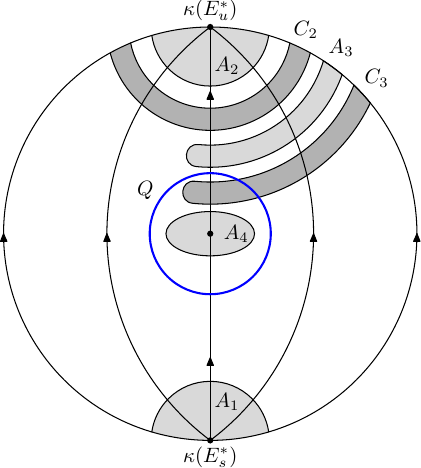}
\caption{A schematic representation of the flow on $ \overline T^* X $.
Different regions (we denote by $ \bullet_j $ the region of microlocalization 
of $ \bullet_j $; control in $ C_j $ is needed for the estimate in $A_j$) 
in which different propagation 
results are applied: for $ A_1 $ we use the radial source
estimates (Proposition \ref{2.6});  for $ A_2 $ the 
radial sink estimates (Proposition \ref{2.7});
for $ A_3 $ the
standard propagation result (Proposition \ref{2.5})
applied to the conjugated operator;
for $ A_4 $ we use elliptic estimates (Proposition \ref{2.4}).
Since for $ A_1 $ and $ A_4 $ we do not need any initial control
(given by $ C_j $), $ C_3 $ can be dynamically 
controlled by regions of the type $ A_1 $ and 
$ A_4 $, and $ C_2 $ is a region of the type $ A_2 $, a partition of unity provides a global 
estimate \eqref{eq:global}.}
\label{f}
\end{figure}

We now have to consider the case of $ A_3\in \Psi^0_h $ microlocalized away from 
$ \kappa (E_u^*)\cup\kappa(E_s^*) $. For that we need to see that the
conjugated operator satisfies the assumptions of the 
Duistermaat--H\"ormander propagation theorem (Proposition \ref{2.5}).
As in \eqref{eq:conj} we have
\[ \widetilde P_{\epsilon, s} ( \lambda ) := e^{ s G ( h ) } \widetilde 
P_{\epsilon} ( \lambda )e^{-sG(h)}=\widetilde P_\epsilon(\lambda)  - i h s \frac{i} {h} [ G ( h ) , \widetilde P_\epsilon ( \lambda )  ]
+ \mathcal O (h^2 )_{ \Psi_{h}^{-1+} } ,
\]
where now, as the operators $ \widetilde P_{\epsilon } ( \lambda) $ and
$ G ( h ) $ are uniformly bounded in $ \Psi^{1}_h $ and 
$ \Psi^{0+}_h $, respectively, the error is in $ \Psi_h^{-1+}$. Hence we
have 
\[ \sigma_h ( P_{ \epsilon , s } ( \lambda ) ) = p_{\epsilon, s} ( x, \xi)  - 
i q_{\epsilon, s } ( x, \xi ) \mod( h \Psi^{0}_h )  \]
where,  with $ p ( x, \xi ) := \xi ( V_x ) $, away from $\xi=0$ we can take
\begin{equation}
\label{eq:pses}  \begin{split}
& p_{\epsilon, s } ( x ,\xi ) = p ( x, \xi) 
- hs \log | \xi|_g H_{m_G }\left( |\xi|_g\chi_1\Big({\varepsilon^2\over h^2}|\xi|^2_g\Big){\varepsilon\over h}|\xi|_g \right) , 
\\
& q_{\epsilon, s} ( x ,\xi) = \chi ( |\xi|_g^2 ) + 
|\xi|_g\chi_1\Big({\varepsilon^2\over h^2}|\xi|^2_g\Big){\varepsilon\over h}|\xi|_g 
- hs \log|\xi|_g H_p m_G ( x, \xi ) \geq 0 . 
\end{split} \end{equation}

We note that $ \tilde p := p_{\epsilon, s } = p + 
\mathcal O ( h)_{ S^{0+}} $ satisfies 
the assumptions of Proposition \ref{2.5} with $ \delta = 1$
and $ q_{\epsilon , s } \leq 0 $. Hence the propagation estimate\eqref{e:hyperbolic-est}
applies.  

As in the proof of~\cite[Proposition 3.4]{dz}, combining~\eqref{eq:A1}, \eqref{eq:A2}, 
Proposition~\ref{2.5}, and the elliptic estimate (Proposition 
\ref{2.4})
we obtain uniformly in $\epsilon$,
\begin{gather}
\label{eq:global}
\begin{gathered}\| u \|_{ H_{s G( h ) } } \leq C h^{-1} \| \widetilde P_{ \epsilon } ( \lambda ) u \|_{
H_{s G ( h ) } } + {\mathcal O} ( h^\infty ) \| u \|_{ H_{h}^{-N} } , 
\ \ 0 < h <  h_0( R ) \\ 
s > s_0 ( R ),  \ \  \lambda \in D ( 0 , R ) , \ \ 
0 < \epsilon \leq h ,
\end{gathered}
\end{gather}
for any $ N $ and that implies \eqref{eq:rad}, finishing the proof.
\end{proof}

We now fix $h<h_0$
and apply Lemma \ref{l:rad} to \eqref{e:houston}. That and \eqref{eq:houston}
give
$$
\|u\|_{H_{sG}(h)}\leq Ch^{-N}\|f\|_{H_{sG}(h)}+\mathcal O(h^{-N}
\varepsilon^\infty)\|u\|_{H_{sG}(h)}
$$
and the $\mathcal O(h^{-N}\varepsilon^\infty)$ can be absorbed into the left-hand side
for $\varepsilon/h$ small enough.

We summarize the result of this section in
\begin{prop}
\label{p:Qeps}
Let $ P_\epsilon $ be given by \eqref{eq:Peps} and $ Q $ by \eqref{eq:defQ}. 
Suppose that $ \lambda \in D ( 0 , R ) $ and that $ 0 \leq \epsilon \leq h/C_0 $. 
Then there exist $ h_0 = h_0 ( R )$, $ s_0 = s_0 ( R ) $, 
(independent of $ \epsilon $) such that for $ 0 < h < h_0 $ and $ s > s_0 ( R ) $
\[   h P_\epsilon - i Q - h \lambda : H_{s G ( h ) }^2 \to H_{s G ( h ) } , \]
is invertible and for some constants $ C $ and  $N$ independent of $ \epsilon $, 
\begin{equation}
\label{eq:Qeps}
\| ( h P_\epsilon - i Q - h \lambda )^{-1} \|_{ H_{ s G ( h ) } \to 
H_{s G ( h ) } } \leq C h^{-N} . 
\end{equation}
\end{prop}
\noindent\textbf{Remark}. Same statement is true if we replace the spaces $H_{sG(h)}$ with
$H^r_{sG(h)}$ for some fixed $r$. Indeed, this amounts to replacing
$sm_G$ by $sm_G+r$ in the weight $G$. The proof of Lemma~\ref{l:ell} remains unchanged.
As for Lemma~\ref{l:rad}, its proof uses the inequality $H_p m_G\leq 0$ (which is still true),
as well as the fact that $H_{sG(h)}$ is equivalent to $H^s_h$ microlocally near $E_s^*$ and
to $H^{-s}_h$ microlocally near $E_u^*$. The space $H^r_{sG(h)}$ is equivalent
to $H^{r+s}_h$ near $E_s^*$ and to $H^{r-s}_h$ near $E_u^*$; for $s$ large enough depending
on $r$ and $R$, Lemma~\ref{l:rad} still holds.

\section{Stochastic approximation of Pollicott--Ruelle resonances}
\label{s:stPR}

In this section we prove Theorem~\ref{th:1}.
Using Proposition \ref{p:Qeps} we see that for $ \lambda \in D ( 0 , R ) $, 
we have the following expression for the meromorphic continuation of
the resolvent of $ P_\epsilon $:
\begin{equation}
\label{eq:merPeps}
( P_\epsilon - \lambda )^{-1} = 
h( h P_\epsilon - i Q - h \lambda)^{-1}  ( I + K ( \lambda,\epsilon ) )^{-1} 
: H_{ s G  } \to H_{s G }, 
\end{equation}
where
\begin{equation}
  \label{e:defKle}
K(\lambda,\varepsilon ) : =iQ( hP_\varepsilon-iQ- h\lambda)^{-1}:H_{sG}\to H_{sG},
\end{equation}
is of trace class and depends holomorphically on $ \lambda $
-- see \eqref{eq:QtoC}. Here $ 
0 < h < h_0 $,  $ 0 \leq \epsilon \leq \varepsilon_0 := h/C_0$ and $ s > s_0$ with $ h_0 $ and 
$ s_0 $ depending on $ R$. We fix $h$ and drop it in the notation for $ H_{s G }$.

As in Proposition \ref{p:detPR} we see that the 
spectrum of $ P_\epsilon $ in $ D ( 0 , R ) $ is given (with multiplicities)
by the zeros of the following Fredholm determinant:
\begin{equation}
D_R(\lambda , \epsilon ):={\det}_{H_{sG}} (I+K(\lambda , \epsilon)).
\end{equation}
Note that, since $Q$ is compactly microlocalized, $K(\lambda,\epsilon)$ acts
$H_{sG}\to H^N$ for all $N$. It follows that $D_R(\lambda,\epsilon)$ is equal to
the $H^N$ determinant of $I+K(\lambda,\epsilon)$ for each $N\geq s$.

To analyze the determinant $D_R(\lambda,\epsilon)$, we apply the following two lemmas. 
We use the notation $ f \in C^1 ( [ a , b ]) $ to mean that 
$f$ and its derivative $f'$ are continuous in $ [ a, b ] $;
here $f'(a),f'(b)$ are the left and right derivatives
of $ f $ at those points. By induction we then define $ C^k ( [ a, b ] ) $
and $ C^\infty ( [ a , b] ) $.

\begin{lemm}
\label{l:Keps}
Let $ R $ and $ h $ be fixed so that \eqref{eq:merPeps} is valid.
Then for every $ k $ there exists $ s_1 = s_1 (k, R ) $ such that for $ s \geq s_1 $,
\begin{equation}
\label{eq:Kla}   K ( \lambda , \epsilon ) \in C^k \left( [0, \epsilon_0 ]_\epsilon , 
\Hol\big( D ( 0 , R )_\lambda , \mathcal L ^1 ( H^s , H^s )\big)\right)  ,\end{equation}
where $ H^s = H^s ( X) $ are Sobolev spaces and $\mathcal L^1$ denotes 
the space of trace class operators.
\end{lemm}
\begin{proof}
We first show that the identity
\begin{equation}
  \label{e:diffid}
\partial_{\varepsilon} (hP_\epsilon-iQ-h\lambda)^{-1}=-ih(hP_\epsilon-iQ-h\lambda)^{-1}\Delta_g
(hP_\epsilon-iQ-h\lambda)^{-1}
\end{equation}
is true for $\varepsilon\in [0,\varepsilon_0]$
in the space $\Hol(D(0,R),\mathcal B(H_{sG}^r,H_{sG}^{r-4}))$,
for each $r$ and for $s$ large enough depending on $R$ and $r$. Here
$\mathcal B$ stands for the class of bounded operators with operator norm.
Indeed, for each $\varepsilon,\varepsilon'\in [0,\varepsilon_0]$,
\begin{equation}
  \label{e:resid}
\begin{gathered}
{(hP_\epsilon-iQ-h\lambda)^{-1}-(hP_{\epsilon'}-iQ-h\lambda)^{-1}\over\epsilon-\epsilon'}\\
=-ih(hP_\epsilon-iQ-h\lambda)^{-1}\Delta_g(hP_{\epsilon'}-iQ-h\lambda)^{-1}
\end{gathered}
\end{equation}
where the right-hand side of the equation is uniformly bounded in $\varepsilon,\varepsilon'$
as an operator $H_{sG}^r\to H_{sG}^{r-2}$. Here we used 
\[ (hP_{\epsilon'}-iQ-h\lambda)^{-1} \in \mathcal B ( H_{sG}^r, H_{sG}^r), \ \
(hP_\epsilon-iQ-h\lambda)^{-1} \in \mathcal B ( H_{sG}^{r-2}, H_{sG}^{r-2}) , \]
(see Proposition~\ref{p:Qeps} and the remark following it)
and the fact that 
$\Delta_g$ is bounded $H_{sG}^r\to H_{sG}^{r-2}$. Now, \eqref{e:resid}
implies that $(hP_\epsilon-iQ-h\lambda)^{-1}$ is Lipschitz (and thus continuous)
as an operator $H_{sG}^r\to H_{sG}^{r-2}$. Passing to the limit $\epsilon'\to\epsilon$
in~\eqref{e:resid}, we obtain~\eqref{e:diffid} in the class
$\mathcal B(H_{sG}^r,H_{sG}^{r-4})$. Holomorphy in $\lambda$ follows automatically
from the holomorphy of each of the operators involved.

Iterating~\eqref{e:diffid}, we see that for each $r$, each $k>0$, and for $s$ large enough
depending on $R,r$ and $k$,
$$
(hP_\epsilon -iQ-h\lambda)^{-1}\in C^k\big([0,\epsilon_0]_\epsilon,\Hol\big(D(0,R),\mathcal B(H_{sG}^r,
H_{sG}^{r-4k})\big)\big).
$$
To obtain \eqref{eq:Kla} we recall 
the definition~\eqref{e:defKle} of $K(\lambda,\epsilon)$, take $ r = 0 $, 
note that $H^s$ embeds into $H_{sG}$ and that the operator
$Q$ is compactly microlocalized and thus of trace class $H_{sG}^{-4k}\to H^s$.
\end{proof}

\begin{lemm}
\label{l:dets}
Suppose that $ \{ X_j \}_{ j=0}^\infty $ is a nested family of Hilbert spaces, $ X_{j+1} 
\subset X_j $. Let  
\begin{equation}
\label{eq:Keps1}  K ( \epsilon ) : X_j \to \bigcap_{\ell=0}^\infty X_{\ell} ,
\end{equation}
be a family of operators such that 
$ K \in C^k ( [ 0 , \epsilon_0 ], {\mathcal L }^1 ( X_k , X_k ) ) $.
Then 
\[  F ( \epsilon ) := {\det}_{ X_0 } ( I + K ( \epsilon ) ) \in C^\infty 
( [ 0 , \epsilon_0 ] ) . \]
\end{lemm}
\begin{proof} 
Because of \eqref{eq:Keps1} we see that $ 
{\det}_{X_j } ( I + K ( \epsilon ) ) $ is independent of $ j$ and
hence we only need to prove that 
$ {\det}_{X_j } ( I + K ( \epsilon ) )  \in C^j 
( [ 0 , \epsilon_0 ]_\epsilon ) 
$,
for any $ j $.  For $ j = 1 $ we note that
$  \partial_\epsilon F ( \epsilon ) = F ( \epsilon ) \tr_{X_1}  \left(( I + K ( \epsilon))^{-1}  \partial_\epsilon K ( \epsilon ) \right)$.
The operators  $  \epsilon \mapsto F ( \epsilon ) ( I + K( \epsilon ))^{-1} $ 
form a continuous family of uniformly bounded operators (see for instance
\cite[(B.7.4)]{res}). Hence, 
$ | \partial_\epsilon F ( \epsilon ) | \leq 
C \| \partial_\epsilon K ( \epsilon ) \|_{ \mathcal L^1 ( X_1, X_1 ) }$.
Higher order derivatives are handled similarly and smoothness of $ F $
follows.
\end{proof}

Applying this Lemma with $ X_j = H^{s_1 ( j , R )  }$ where $ s_1 $ comes
from Lemma \ref{l:Keps}, we see that 
$ \epsilon \mapsto D_R ( \lambda , \epsilon )  $ is a smooth
function of $ \epsilon \in [ 0 , \epsilon_0 ]$ with values in $ 
{\rm{Hol}} ( D  ( 0 , R ) ) $.  Rouch\'e's theorem 
implies that the zeros are continuous in $\epsilon$ up to $ 0 $, proving Theorem~\ref{th:1}. If $ \mu_0 $ is a simple zero of $ D_R ( \lambda, 0 ) $ then 
for $ 0 \leq \epsilon < \epsilon_1$,  $ D_R ( \lambda , \epsilon ) $
has a unique zero, $ \mu ( \epsilon ) $,
close to $ \mu_0 $. Smoothness of $ D_R $ in $ \epsilon $ shows that
\[  \mu ( \epsilon ) \in C^\infty ( [ 0 , \epsilon_1 ] ) . \]

When the zeros are not simple (in particular, when the eigenvalues
of $ P_0 $ are not semisimple) the situation is potentially quite 
complicated.
 However we have smoothness of spectral projectors:
\begin{prop}
\label{p:specP}
Suppose that $ \mu_0\in D(0,R-1) $ is an eigenvalue of 
$ P_0 : H_{ s G } ( X ) \to H_{s G } ( X) $, $ s \geq s_0(R) $,
and that multiplicity of $ \mu_0 $ is $ m$:
\[  m = \tr \Pi_0 \,, \ \ \  \Pi_0 = \frac 1 { 2 \pi i } \oint_{\gamma_\delta} 
( \lambda - P_0 )^{-1} d\lambda, \]
where $ \gamma_\delta : [0, 2 \pi ) \ni t \to \mu_0 + \delta e^{ it } $, 
and $ \delta $ is small enough.

Then there exists $ \epsilon_0 $ and $ \delta$ such that 
for $ 0 < \epsilon \leq \epsilon_0 $, $ P_\epsilon $ has exactly 
$ m $ eigenvalues in $ D ( \mu_0 , \delta) $:
\begin{equation}
\label{eq:specP1}
\tr \Pi_\epsilon = m , \ \  \Pi_\epsilon := \frac 1 { 2 \pi i } \oint_{\gamma_\delta} 
( \lambda - P_\epsilon )^{-1} d\lambda, \ \ \Pi_\epsilon^2 = \Pi_\epsilon, 
 \end{equation}
and $ \Pi_\epsilon \in C^\infty ( [ 0 , \epsilon_0 ] , 
{\mathcal L }^1 ( C^\infty ( X ) , \mathcal D' ( X) ) $.
More precisely, the projections $ \Pi_\epsilon $ have rank $ m $ and 
for each $ j $ there exists $ s_j $ such that
\begin{equation}
\label{eq:specP}  \Pi_\epsilon \in C^j \big( [0, \epsilon_0 ], {\mathcal L} ( H_{s_j G } , 
H_{ s_j G} )\big) \subset 
C^j \big( [0, \epsilon_0 ], {\mathcal L} ( H^{-s_j } , 
H^{ s_j} ) \big). \end{equation}
\end{prop}
\begin{proof}
From the analysis of the determinants, we already know that
there exist $ \epsilon_0 , \delta $ such that for 
$ 0 \leq \epsilon \leq \epsilon_0 $, ${\lambda} \mapsto D_{R } ( \lambda, \epsilon ) 
$ has no zeros on $ | \lambda - \mu_0 | = \delta $ and has exactly $ m $ zeros in 
$ D ( \mu_0 , \delta ) $. Hence the spectral projectors are well defined by 
the formula in \eqref{eq:specP1} and their rank is  equal to $ m$. 
To consider regularity, we choose $ h $ sufficiently small (depending on 
$ R $) and write
$$
\begin{gathered}
(P_\varepsilon-\lambda)^{-1} = (P_\epsilon -ih^{-1}Q-\lambda)^{-1} - 
i h^{-1} (P_\varepsilon-\lambda)^{-1} Q (P_\epsilon -ih^{-1}Q-\lambda)^{-1}.
\end{gathered}
$$
Since the first term is holomorphic in $ \lambda \in D ( 0 , R ) $
we have 
\[ \Pi_\epsilon := -\frac 1 { 2 \pi h } \oint_{\gamma_\delta} 
( \lambda - P_\epsilon )^{-1} 
Q ( P_\epsilon - i h ^{-1} Q - \lambda )^{-1}  d\lambda . \]
Also 
\[  ( \lambda - P_\epsilon )^{-1} = {\mathcal O}_{ R, r, s } ( 1 ) :
H^{r}_{ s G } \to H^r_{ s G } , \ \  s \geq s_0 ( R, r ) , \ \ 
\lambda \in \partial D( \mu_0 , \delta ) . \]
Hence the same argument as in the proof of Lemma \ref{l:Keps}
shows $j$-fold differentiability of $ \Pi_\epsilon $ as bounded operators 
 $ H_{ s_j G } \to H_{ s_j G }$.
\end{proof}

\section{Stochastic stability in the case of contact Anosov flows}
\label{s:con}

We now turn to the proof of Theorem \ref{th:2}. The first result concerns values of $ \epsilon $ larger than $h^2$.
Here we do not need to make the contact assumption on the flow.
\begin{lemm}
\label{l:ell1}
Let $ P_\epsilon $ be given by \eqref{eq:Peps}. 
There exist $ K_0 > 0 $ and 
$ h_0 > 0 $ 
such that for any $ \gamma > 1 $, $ h $ and $ \epsilon $
satisfying 
\[  0 <  K_0 \gamma h^2 < \epsilon 
,  \ \ 0 < h < h_0 ,  \]
we have
\begin{equation}
\label{eq:ell2}
(h P_\epsilon - z)^{-1} = {\mathcal O} \left( \frac1 {\sqrt \epsilon} \right) : L^2 ( X ) 
\to L^2 ( X) , \ \ z \in [ \textstyle{\frac12, \frac32}] - i  [ 0, \gamma h ] . 
\end{equation}
In particular $ h P_\epsilon $ does not have any spectrum 
with $ |z - 1| < \frac12 $ and $ \Im z > - \gamma h $.
\end{lemm}

\noindent
{\bf Remark.} The lemma shows that for any fixed $ \epsilon $ the number 
of eigenvalues of $ P_\epsilon $ in $ \Im \lambda > - C $ is finite. In fact
the rescaling from $ z $ to $ \lambda $ shows that there are only 
finitely many eigenvalues of $ P_\epsilon $ in $\{ \Im \lambda > - \epsilon 
| \Re \lambda|^2 / C_0\} $, for some fixed $ C_0 $. This leads to an 
easy justification of the expansion \eqref{eq:t4}. 
We also see that a gap $ \Im z > - \gamma h $ for any $ \gamma $ is valid for 
$ \epsilon > C(\gamma) h^2 $. Hence in what follows we will assume that
$ \epsilon = \mathcal O ( h^2 ) $.

\begin{proof}
We fix the volume form on $X$ induced by the metric $g$, so that
the operator $\Delta_g$ is symmetric on $L^2(X)$. Take
$u\in H^2(X)$ and denote $f:=(hP_\epsilon-z)u$; then
$$
\langle f,u\rangle_{L^2}=\left\langle \left( \textstyle{\frac h i}V+ih\epsilon\Delta_g-z\right) u,u\right\rangle_{L^2}
=\textstyle {\frac h  i}\langle Vu,u\rangle_{L^2}-ih\epsilon\|\nabla_g u\|_{L^2}^2-z\|u\|_{L^2}^2.
$$
Taking the real part, we get
$$
\Re\langle f,u\rangle_{L^2}=h\Im\langle Vu,u\rangle_{L^2}- \Re z\|u\|_{L^2}^2.
$$
Since $\Re z\geq {1\over 2}$ and $V$ is a vector field, we find for some constant $C$
independent of $h,z,\epsilon$,
$$
\|u\|_{L^2}^2\leq C\|f\|_{L^2}\cdot\|u\|_{L^2}+Ch\|\nabla_g u\|_{L^2}\cdot \|u\|_{L^2},
$$
which implies
\begin{equation}
  \label{e:yippie1}
\|u\|_{L^2}\leq C\|f\|_{L^2}+Ch\|\nabla_g u\|_{L^2}.
\end{equation}
Now, taking the imaginary part, we get for $F:={1\over 2}\Div V\in C^\infty(X)$,
$$
\Im\langle f,u\rangle_{L^2}=h\langle Fu,u\rangle_{L^2}-h\epsilon\|\nabla_g u\|^2_{L^2}-(\Im z)\|u\|_{L^2}^2.
$$
Since $\Im z\geq -\gamma h$ and $F$ is a bounded function, we get
$$
h\epsilon\|\nabla_g u\|_{L^2}^2\leq C\|f\|_{L^2}\cdot \|u\|_{L^2}+(C+\gamma)h\|u\|_{L^2}^2,
$$
which implies
\begin{equation}
  \label{e:yippie2}
\|\nabla_g u\|_{L^2}\leq Ch^{-1}\epsilon^{-1/2}\|f\|_{L^2}+(C+\sqrt\gamma)\epsilon^{-1/2}\|u\|_{L^2}.
\end{equation}
Combining~\eqref{e:yippie1} and~\eqref{e:yippie2}, we get
$$
\|u\|_{L^2}\leq C\epsilon^{-1/2}\|f\|_{L^2}+Ch(C+\sqrt\gamma)\epsilon^{-1/2}\|u\|_{L^2}.
$$
For $K_0$ large enough and $\epsilon>K_0\gamma h^2$, $\gamma>1$, we have
$Ch(C+\sqrt\gamma)\epsilon^{-1/2}<{1\over 2}$, implying~\eqref{eq:ell2}.
\end{proof}

To prove Theorem~\ref{th:2} we follow~\cite{nz5}. We first
prove a result in which damping is introduced near the fiber infinity
in  $T ^* X$. For that we introduce a complex absorbing operator
\begin{equation}
\label{eq:defW0}   W_0:=   - f (-  h^2 \Delta_g  )  h^2 \Delta_g , \ 
\end{equation}
where $ f \in C^\infty ( \RR )$ satisfies the following conditions:
\begin{equation}
\label{eq:condf}  f \geq 0 ,  \ \  |f^{(k)}|  \leq C_k f^{1 - \alpha } ,  \ \ 
 f ( t) \equiv 0 \text{ for }t\leq C_0, \ \ 
 f(t) \equiv 1 \text{ for }t\geq 2C_0
\end{equation}
for some $ \alpha < \frac12 $ and some large constant $C_0$.
The technical condition on $ f^{(k) }$ is useful for 
comparing the propagators of $ \widehat P_\epsilon $ and $ P_0 $ -- 
see \cite[Appendix]{nz5}.

 With $ P_\epsilon $ given by \eqref{eq:Peps} we now
consider
\begin{equation} 
\label{eq:wideP} \widehat P_\epsilon := hP_\epsilon - i W_0 . 
\end{equation}
Unlike in \S\S \ref{s:mimoa},\ref{s:stPR} we will now work near a fixed rescaled
energy level $ z = h \lambda \sim 1 $ rather than near the zero energy.

The next result is an almost immediate application of \cite[Theorem 2]{nz5}:

\begin{lemm}
\label{l:nz5}
Suppose that the flow $ \varphi_t : X \to X $ is a {\em contact} Anosov flow
(see~\eqref{eq:contact}), 
$ \widehat P_\epsilon $ is given by \eqref{eq:wideP} and that
$ \epsilon = \mathcal O ( h^2 ) $. Let $\gamma_0  $ be 
the averaged Lyapounov exponent defined in \eqref{eq:gammd}. 
Then for any $ \delta  > 0 $ and $ s $ there exist $ h_0 $, $ c_0 $,
$ C_1 $,
such that for $ 0 < h < h_0 $,
\begin{equation}
\label{eq:nz5}
  \|  ( \widehat P_\epsilon - z )^{-1} \|_{L^2 \to L^2 }
\leq C_1 
  h^{-1  + c_0 \Im z / h } \log (1/h) , 
  \end{equation}
for 
\begin{equation}
  \label{e:nz5c}
  z \in [  \textstyle{{1\over 2}, {3\over 2}}] -
i h [ 0 , \gamma_0/ 2 - \delta ] \,.
\end{equation}
\end{lemm}

\noindent
{\bf Remark.} The bound \eqref{eq:nz5} is more precise than the bound
\eqref{eq:t2} which corresponds to $ \mathcal O ( h^{-N} ) $. It is obtained
by interpolation between the bound $ 1/\Im z $ for $ \Im z > 0 $ and
the polynomial bound $ \mathcal O  ( h^{-N} ) $ -- see \cite[Lemma 4.7]{Bu}, \cite[Lemma 2]{tz}. Using the fact that
$ H_{ s G } $ are complex interpolation spaces \cite{ca} the estimate
\eqref{eq:t2} can be refined to a form similar to \eqref{eq:nz5}.

\begin{proof} 
Put $ W = - (\epsilon/h) h^2 \Delta_g + W_0 $ (where $ W_0 $ appearing in the
definition of $ \widehat P_\epsilon $ is given by \eqref{eq:defW0}).
Note that $\widehat P_\epsilon=hP_0-iW$.
We have $ W \in \Psi_h^2 ( X ) $, $ W \geq 0 $, 
and since $ \epsilon = {\mathcal O} ( h ^2 ) $, 
\[ w:= \sigma_h ( W ) = \sigma_h ( W_0 ) = f ( |\xi|^2_g ) |\xi|_g^2 . \]
Hence $ P := h P_0 $ and $ W $ satisfy the assumptions \cite[(1.9),(1.10)]{nz5}.
The only difference is that $ P \in \Psi^1_h ( X ) $, so that in the 
notation of \cite{nz5}, $ k = 2 $ and $ m = 1$. Replacing $ k $ with 
 $m $ in the ellipticity condition \cite[(1.9)]{nz5} does not change
the proofs in \cite{nz5} (in particular it does not affect 
\cite[Proposition A.3]{nz5}): all the arguments are microlocal 
near the (compact) trapped set 
\begin{equation}
\label{eq:wideK}  \widehat K := \{ ( x, \xi ) : | p_0 ( x, \xi ) - 1 |< 1/2, \ \ 
\exp ( t H_{p} ) ( x, \xi ) \not \to \infty , t \to \pm \infty \}, 
\end{equation}
$ p ( x, \xi ) = \xi ( V_x ) $.

Since $ \varphi_t $ is a contact Anosov flow, 
the trapped set is normally hyperbolic in the sense of \cite[(1.14)--(1.17)]{nz5}
-- see \cite[\S 9]{nz5}. Hence we can apply \cite[Theorem~2]{nz5} and
obtain the bound \eqref{eq:nz5}.
\end{proof} 

We are now ready for
\begin{proof}[Proof of Theorem \ref{th:2}.]
We first note that \eqref{eq:t2} follows by rescaling $
z = h \lambda $  from a semiclassical 
estimate between the weighted spaces (we recall that 
$ H^s \subset H_{ s G( 1 ) } \subset H^{-s }$)
\begin{equation}
\label{eq:precise}   ( h P_\epsilon - z )^{-1} = \mathcal 
  O (  h^{-N} ) : H_{ s_0  G( 1 ) }  ( X ) 
  \to H_{s_0  G ( 1 )  } ( X )  . \end{equation}
for the same range of $ z $'s as in \eqref{e:nz5c}. 
By Lemma~\ref{l:ell1},
we can assume that $ \epsilon = \mathcal O ( h^2 ) $. 

To prove \eqref{eq:precise}
we follow the strategy as in \cite[\S 9]{nz5} 
combined with the estimates
of~\S\ref{s:mimoa}. For that we choose $ Q $ in \eqref{eq:defQ} 
and $ W_0 $ in \eqref{eq:defW0}, so that
for the weight 
$ G $ in \eqref{eq:weighted} and the trapped set $ \widehat K $ defined in \eqref{eq:wideK} we have 
\[ \WFh ( Q ) \cap \WFh ( G )=
 \widehat K \cap \WFh ( I - Q ) = \WFh ( Q ) \cap 
 \WFh ( W_0 ) = \emptyset .  \]
Since $ \widehat K $ is compact that is possible by modifying the 
conditions on $ \chi $ in \eqref{eq:defQ} and by increasing $ C_0 $ in 
\eqref{eq:condf}. 

To stay close to the notation of \cite[\S 9]{nz5} we now put 
$ P_\infty := h P_0 + i \epsilon h \Delta_g - i Q $.
To apply the gluing argument of Datchev--Vasy \cite{DaVa} as in 
\cite[\S 9]{nz5} we check that the conclusions of \cite[Lemma 9.19]{nz5}
are valid. First,
\begin{equation}
\label{eq:Pinfty1}   ( P_\infty - z )^{-1} = {\mathcal O} (h^{-N_1} ) 
: H_{s_0 G ( 1 ) } \to H_{s_0 G ( 1 ) } , \ \ \Im z > - \gamma_0/2  , \ \ 
| \Re z  - 1 | < 1/2 , 
\end{equation}
is proved similarly as \eqref{eq:Qeps}. Indeed, Lemma~\ref{l:ell} holds
for $\lambda=\mathcal O(\varepsilon^{-1/2})$,
and this condition is true since $\lambda=\mathcal O(h^{-1})$ and $\varepsilon=\mathcal O(h^2)$.
The proof of Lemma~\ref{l:rad} goes through as in the case $\lambda=\mathcal O(1)$. The proof
of~\eqref{eq:Qeps} works as before, using again that $\varepsilon=\mathcal O(h^2)$.

Next, we need the
propagation statement,
\[  \begin{split} & u = (P_\infty  - z)^{-1} f , \ \ \WFh(f)\cap\partial\overline T^*X=\emptyset  \
  \ \Longrightarrow  \\
& \ \ \ \ \ \ \ \  \WFh ( u ) \setminus (\WFh ( f )\cup \partial\overline T^*X) \subset \exp ( [ 0 , \infty 
) H_{p}  ) \left(
  \WFh ( f ) \cap p^{-1} ( \Re z  ) \right)  ,  
\end{split}
\]
where $ p ( x, \xi ) = \xi ( V_x ) $. This statement follows from the propagation 
theorems reviewed in the proof of Lemma~\ref{l:rad}; note that
$(hP_\epsilon-\widetilde P_\epsilon(0))u=\mathcal O(h^\infty)_{C^\infty}$
since $\epsilon=\mathcal O(h^2)$ and $\WFh(f)$ does not intersect the fiber infinity.

We can now follow the gluing argument of the resolvent
estimates on  $ ( P_\infty - z)^{-1} $
and $ ( \widehat P_\epsilon - z )^{-1} $ (given in \eqref{eq:nz5})
as in \cite[\S 9]{nz5}  to obtain \eqref{eq:precise}. In the notation 
of \cite[\S 9]{nz5} the parametrix for $ ( h P_\epsilon - z)^{-1} $
is given by 
\[  F ( z ) := A_1 ( \widehat P_\epsilon - z )^{-1} A_0 + 
B_1 ( P_\infty - z )^{-1} B_0 , \]
where $ A_j , I - B_j \in \Psi^{\rm{comp}}_h ( X ) $ are suitably 
chosen, with $ \WFh ( A_j ) \cap \WF ( G ) = \emptyset $. 
Away from the microsupport of $ G $, the spaces $ H_{sG ( h ) } $ are
microlocally equivalent to $ L^2 $. Hence
the $ L^2 $ estimates on  $ ( \widehat P_\epsilon - z)^{-1} $ imply the
$ H_{sG ( h ) } $ estimates on $ F ( z ) $. The gluing argument of 
\cite{DaVa} as recalled in \cite[\S 8]{nz5} concludes the proof of 
\eqref{eq:precise}.
\end{proof}

\def\arXiv#1{\href{http://arxiv.org/abs/#1}{arXiv:#1}}

\end{document}